\documentclass[11pt, reqno, psamsfonts]{amsart}
\pdfoutput=1
\usepackage{amssymb}
\usepackage{amsthm}
\usepackage{amsmath}
\usepackage{latexsym}
\usepackage[english]{babel}
\usepackage{graphicx}
\usepackage{wrapfig}
\usepackage[justification=centering, labelfont=bf]{caption}
\usepackage{subcaption}
\usepackage{mathtools}
\usepackage[hidelinks]{hyperref}
\usepackage{enumitem}
\usepackage{mathrsfs}
\usepackage{thmtools}
\usepackage{stmaryrd}
\usepackage[oldstylenums]{kpfonts}
\usepackage[shortcuts]{extdash}
\usepackage{framed}
\usepackage{centernot}
\usepackage[stable]{footmisc}
\usepackage[foot]{amsaddr}
\usepackage[backend=biber, style=alphabetic, sorting=nyt, maxnames=100]{biblatex}
\usepackage[left=2.5cm,right=2.5cm,top=2.5cm,bottom=2.5cm,bindingoffset=0cm]{geometry}

\title{On Baire Measurable Colorings of Group Actions}
\date{}
\author{Anton~Bernshteyn}
\address{Department of Mathematics, University of Illinois at Urbana--Champaign, IL, USA and Department of Mathematical Sciences, Carnegie Mellon University, Pittsburgh, PA, USA}
\email{abernsht@math.cmu.edu}
\thanks{This research is partially supported by the Illinois Distinguished Fellowship.}

\newtheorem{theo}{Theorem}[section]
\newtheorem{prop}[theo]{Proposition}
\newtheorem{lemma}[theo]{Lemma}
\newtheorem{corl}[theo]{Corollary}

\newtheorem{claim}{Claim}[theo]

\theoremstyle{definition}
\newtheorem{defn}[theo]{Definition}

\theoremstyle{remark}
\newtheorem*{remk}{Remark}

\newcommand*{\myproofname}{Proof}
\newenvironment{claimproof}[1][\myproofname]{\begin{proof}[#1]}{\end{proof}}

\newcommand{\0}{\varnothing}
\newcommand{\set}[1]{\{#1\}}
\newcommand{\dom}{\mathrm{dom}}

\newcommand{\supp}{\mathrm{supp}}

\newcommand{\acts}{\curvearrowright}
\newcommand{\N}{\mathbb{N}}
\newcommand{\Z}{\mathbb{Z}}
\newcommand{\R}{\mathbb{R}}
\newcommand{\Rpos}{[0;\infty)}
\newcommand{\Q}{\mathbb{Q}}
\renewcommand{\P}{\mathfrak{I}}

\renewcommand{\epsilon}{\varepsilon}
\renewcommand{\phi}{\varphi}
\renewcommand{\theta}{\vartheta}
\renewcommand{\tilde}{\widetilde}
\renewcommand{\leq}{\leqslant}
\renewcommand{\geq}{\geqslant}

\newcommand{\symdif}{\triangle}

\newcommand{\Sol}{\mathbf{Col}}
\newcommand{\finset}[1]{[#1]^{<\infty}}
\newcommand{\finfun}[2]{[#1 \to #2]^{<\infty}}
\newcommand{\pfun}[2]{[#1 \to #2]}
\newcommand{\pcol}[3]{[#1 \to #2]^{#3}}

\newcommand{\rest}[2]{{#1}\vert{#2}}
\newcommand{\G}{\Gamma}
\let\secsign\S
\renewcommand{\S}{\mathbf{Sh}}
\newcommand{\PrCol}{\mathbf{PrCol}}

\newcommand{\A}{\Sigma}
\newcommand{\defeq}{\coloneqq}
\newcommand{\stareq}{=^\ast}
\newcommand{\BM}{{\operatorname{BM}}}
\newcommand{\subshift}{\Omega}
\newcommand{\FPC}[1]{\mathbf{Fin}({#1})}
\newcommand{\Ball}[2]{\mathbf{Ball}(#1, #2)}
\newcommand{\dist}{\mathbf{dist}}
\newcommand{\loc}[2]{\mathbf{Loc}_{#2}({#1})}

\renewcommand{\succeq}{\succcurlyeq}
\newcommand{\K}{\mathcal{K}}
\newcommand{\forceset}[2]{{\llbracket {#1}, {#2} \rrbracket}}
\newcommand{\Bairealg}[2]{{\mathfrak{Baire}}}
\newcommand{\E}{\mathcal{E}}

\let\bcup\bigcup
\renewcommand{\bigcup}{\textstyle\bcup\nolimits}

\let\bcap\bigcap
\renewcommand{\bigcap}{\textstyle\bcap\nolimits}

\numberwithin{equation}{section}

\makeatletter
\newcommand{\neutralize}[1]{\expandafter\let\csname c@#1\endcsname\count@}
\makeatother

\renewcommand{\mkbibnamefirst}[1]{\textsc{#1}}
\renewcommand{\mkbibnamelast}[1]{\textsc{#1}}
\renewcommand{\mkbibnameprefix}[1]{\textsc{#1}}
\renewcommand{\mkbibnameaffix}[1]{\textsc{#1}}
\renewcommand{\finentrypunct}{}

\renewbibmacro{in:}{}

\renewbibmacro*{volume+number+eid}{%
	\printfield{volume}%
	\setunit*{\addnbspace}
	\printfield{number}%
	\setunit{\addcomma\space}%
	\printfield{eid}}

\DeclareFieldFormat[article]{volume}{\textbf{#1}\space}
\DeclareFieldFormat[article]{number}{\mkbibparens{#1}}

\DeclareFieldFormat{journaltitle}{#1,}
\DeclareFieldFormat[thesis]{title}{\mkbibemph{#1}\addperiod}
\DeclareFieldFormat[article, unpublished, thesis]{title}{\mkbibemph{#1},}
\DeclareFieldFormat[book]{title}{\mkbibemph{#1}\addperiod}
\DeclareFieldFormat[unpublished]{howpublished}{#1, }

\DeclareFieldFormat{pages}{#1}

\DeclareFieldFormat[article]{series}{Ser.~#1\addcomma}

\begin{document}
	
	\maketitle
	
	\begin{abstract}
		The field of \emph{descriptive combinatorics} investigates the question, to what extent can classical combinatorial results and techniques be made topologically or measure-theoretically well\-/behaved? This paper examines a class of coloring problems induced by actions of countable groups on Polish spaces, with the requirement that the desired coloring be Baire measurable. We show that the set of all such coloring problems that admit a Baire measurable solution for a particular free action $\alpha$ is complete analytic (apart from the trivial situation when the orbit equivalence relation induced by $\alpha$ is smooth on a comeager set); this result confirms the ``hardness'' of finding a topologically well-behaved coloring. When $\alpha$ is the shift action, we characterize the class of problems for which $\alpha$ has a Baire measurable coloring in purely combinatorial terms; it turns out that closely related concepts have already been studied in graph theory with no relation to descriptive set theory. We remark that our framework permits a wholly dynamical interpretation (with colorings corresponding to equivariant maps to a given subshift), so this article can also be viewed as a contribution to \emph{generic dynamics}.
	\end{abstract}
	
	\section{Introduction}
	
	
	\noindent A typical combinatorial problem is that of \emph{coloring}, i.e., assigning to each element of a given structure (a graph, say) an element of some countable set---a ``color''---in a way that fulfills a specific set of constraints; for instance, one might require the colors of adjacent vertices in a graph to be distinct. Classical combinatorics mostly confines itself to studying colorings of finite structures; this restriction does not usually result in any loss of generality, since coloring an infinite structure, via a straightforward compactness argument, can often be reduced to coloring its finite substructures (see, e.g., the classical theorem of de Bruijn and Erd\H{o}s~\cite[Theorem~1]{dBE51} stating that an infinite graph is $k$-colorable if and only if so is each of its finite subgraphs).
	
	However, an infinite structure is sometimes equipped with a topology or a measure, and it is then natural to ask for colorings which not only satisfy the combinatorial constraints, but also behave well topologically or measure-theoretically; with these additional requirements, compactness can no longer be used to directly reduce the problem to the finite case. As a simple concrete example, fix some irrational $\alpha \in (0;1)$ and consider the graph $\mathcal{G}_\alpha$ with vertex set the half-open interval~$[0;1)$ whose edges connect the pairs of vertices $x$, $y$ with $y - x = \pm \alpha \pmod 1$. Since $\alpha$ is irrational, $\mathcal{G}_\alpha$ is a disjoint union of (continuumly many) paths, infinite in both directions; in particular, there is a vertex coloring $f \colon [0;1) \to 2$ of $\mathcal{G}_\alpha$ such that $f(x) \neq f(y)$ whenever $x$ and $y$ are adjacent. Yet, it is not hard to verify that such a coloring can neither be Lebesgue measurable nor Baire measurable with respect to the usual topology on~$[0;1)$ (see, e.g., \cite[p.~2]{CMT-D16}).
	
	Questions regarding colorings that not only have some nice combinatorial properties but also behave well in the sense of topology or measure are studied in the area of \emph{descriptive combinatorics}, which has recently emerged out of interactions between descriptive set theory, combinatorics (graph theory in particular), and some other fields, such as ergodic theory. A~comprehensive state-of-the-art survey of descriptive combinatorics can be found in~\cite{KM16}.
	
	Since a known result in finite combinatorics cannot be simply turned into a theorem of descriptive combinatorics using compactness, the second best option is to analyse the \emph{proof} of the known result to see if \emph{it} can be adapted to the descriptive setting. This turns out to be a rather fruitful approach. For example, in their seminal paper~\cite{KST99}, Kechris, Solecki, and Todorcevic established~\cite[Proposition~4.6]{KST99} that a Borel graph $\mathcal{G}$ with finite maximum degree $d$ admits a Borel coloring using at most $d+1$ colors such that the colors of adjacent vertices are distinct. For a \emph{finite} graph $\mathcal{G}$, such a coloring can be found ``greedily'': One simply considers the vertices of $\mathcal{G}$ one by one and assigns to each vertex the first color not yet used on any of its neighbors; since the total number of colors exceeds the maximum number of neighbors a vertex can have, there is always at least one color available. In their proof of \cite[Proposition~4.6]{KST99}, Kechris, Solecki, and Todorcevic devised a Borel analog of this ``greedy'' algorithm.
	
	Some techniques in finite combinatorics are more amenable to descriptive generalizations (more \emph{constructive}, one could say) than others. For instance, of the ways of obtaining matchings in graphs, the arguments based on augmenting paths appear to be especially well-suited for the purposes of descriptive combinatorics (see, e.g., \cite{EL10, LN11}). A recent family of examples is formed by measurable versions of the Lov\'asz Local Lemma (see \cite{Ber16, CGMPT16}), made possible by the breakthrough result of Moser and Tardos~\cite{MT10}, who found a new, algorithmic proof of the Lov\'asz Local Lemma.
	
	The above examples suggest that some precise correspondences between results in finite and descriptive combinatorics might still be present; the existence of a well-behaved coloring of a certain kind could, perhaps, be equivalent to a purely combinatorial statement such as the existence of a ``greedy''-like algorithm to find it. One of the main results of this article is Theorem~\ref{theo:combi}, which confirms this suspicion for a particular class of coloring problems and a specific notion of well-behavedness, namely Baire measurability (see Definition~\ref{defn:BM_col}).
	
	
	An ample supply of examples in descriptive combinatorics is provided by actions of countable groups, and that is the convenient framework in which we perform our investigation. (For instance, the graph $\mathcal{G}_\alpha$ defined previously is induced by the action of the additive group $\Z$ on $[0;1)$ given by $n \cdot x \defeq x + n\alpha \pmod 1$.) This article therefore can be considered a contribution to \emph{generic dynamics}; see~\cite{Wei00, SW15} for an introduction to the subject.
	
	
	\subsection*{Basic notation and terminology}\label{subsec:basics}
	
	We use $\N \defeq \set{0, 1, \ldots}$ to denote the set of all nonnegative integers and identify each $k \in \N$ with the set $\set{i \in \N \,:\, i < k}$. The sets $\N$ and $k$ for $k \in \N$ are assumed to carry discrete topologies.
	
	We identify a function $f$ with its graph, i.e., with the set $\set{(x, y) \,:\, f(x) = y}$; this enables the use of standard set-theoretic notation, such as $\cup$, $\cap$, $\subseteq$, etc., for functions. In particular, $\0$ denotes the empty function as well as the empty set. For a function $f$ and a subset $S$ of its domain, $\rest{f}{S}$ denotes the restriction of $f$ to $S$. We write $f \colon X \rightharpoonup Y$ to indicate that $f$ is a partial function from $X$ to $Y$, i.e., a function of the form $f \colon X' \to Y$ with $X' \subseteq X$.
	
	Given sets $A$ and $B$,
	\[
		\begin{array}{rll}
			\text{--} & A^B & \text{denotes the set of all functions $f \colon B \to A$;}\\
			\text{--} & \finset{B} & \text{denotes the set of all finite subsets of $B$;}\\
			\text{--} & \finfun{B}{A} & \text{denotes the set of all partial functions $\phi \colon B \rightharpoonup A$ with $\dom(\phi) \in \finset{B}$.}
		\end{array}
	\]
	
	Whenever we use symbols $\max$, $\min$, $\inf$, and $\sup$, they are applied to subsets of the set $\Rpos$ of nonnegative real numbers. In particular, $\inf \0 = \infty$ and $\sup \0 = 0$.
	
	Our standard references for descriptive set theory are~\cite{Kec95} and~\cite{Tse16}.
	
	A separable topological space is \emph{Polish} if its topology is generated by a complete metric. Note that a compact space is Polish if and only if it is metrizable. Most notions related to Baire category make sense for a wider class of topological spaces (the so-called \emph{Baire spaces}); however, to simplify the matters, we will only talk about Polish spaces here. A subset of a Polish space is \emph{meager} if it can be covered by countably many nowhere dense sets; \emph{nonmeager} if it is not meager; and \emph{comeager} if its complement is meager. We say that two sets $A$ and $B$ are \emph{equal modulo a meager set}, or \emph{$\ast$-equal}, in symbols $A \stareq B$, if their symmetric difference $A \symdif B$ is meager. A set is \emph{Baire measurable} if it is $\ast$-equal to an open set.\footnote{Sometimes meager sets are referred to as sets \emph{of first category}; nonmeager---as \emph{of second category}; comeager---as \emph{residual}; and Baire measurable---as having the \emph{property of Baire}.} The meager sets form a $\sigma$-ideal (i.e., meagerness is a notion of smallness), and the Baire measurable sets form a $\sigma$-algebra, which contains all Borel sets (and much more). The cornerstone result of the Baire category theory is the \emph{Baire category theorem}, which asserts that a nonempty open subset of a Polish space is nonmeager; equivalently, the intersection of countably many dense open subsets of a Polish space is dense. For a Baire measurable set $A$ and a nonempty open set $U$, we say that $U$ \emph{forces} $A$, or $A$ is \emph{comeager in} $U$, in symbols $U \Vdash A$, if the difference $U \setminus A$ is meager. The following way of phrasing the Baire category theorem is rather useful:
	\begin{prop}[{\textbf{Baire alternative}\footnote{This name is due to Anush Tserunyan; see~\cite[Proposition~9.8]{Tse16}.}~\cite[Proposition~8.26]{Kec95}}]
		 A Baire measurable subset of a Polish space is either meager, or else, comeager in some nonempty open set.
	\end{prop}
	\noindent For more background on Baire category, see \cite[Section~8]{Kec95} and \cite[Sections~6 and~9]{Tse16}.
	
	A \emph{standard Borel space} is a set $X$ equipped with a $\sigma$-algebra $\mathfrak{B}(X)$ of its \emph{Borel subsets} that coincides with the Borel $\sigma$\=/algebra generated by some Polish topology on $X$. The \emph{Borel isomorphism theorem} \cite[Theorem~15.6]{Kec95} states that a standard Borel space $X$ is either countable, in which case $\mathfrak{B}(X)$ is the power set of $X$, or else, is isomorphic to every other uncountable standard Borel space.
	
	A~function $f \colon X \to Y$ from a Polish space $X$ to a standard Borel space $Y$ is \emph{Baire measurable} if for all Borel $B \subseteq Y$, the preimage $f^{-1}(B)$ is Baire measurable (as a subset of $X$). 
	
	
	
	\section{Main definitions and statements of results}\label{sec:main}
	
	\subsection{Groups, group actions, and their colorings}
	
	Throughout, $\G$ denotes a countably infinite discrete group with identity element $\mathbf{1}$. We fix an arbitrary proper\footnote{Recall that a metric space is \emph{proper} if every closed and bounded subset of it is compact. For discrete spaces, this is equivalent to saying that every ball of finite radius is a finite set.} right-invariant metric $\dist$ on $\G$. Note that such a metric always exists. Indeed, if $\G$ is finitely generated, then $\dist$ could be the word metric with respect to any finite generating set; in general, one can take
	$$
	\dist(\gamma, \delta) \defeq \min\set{i_1 + \ldots + i_k \,:\, \epsilon^{\pm 1}_{i_1} \cdots \epsilon^{\pm 1}_{i_k} \gamma = \delta},
	$$
	where $\set{\epsilon_1, \epsilon_2, \ldots} = \G$ is an enumeration of the elements of $\G$ in some order. Any two proper right-invariant metrics on $\G$ are coarsely equivalent, so the specific choice of the metric will be irrelevant for our purposes. We use $\Ball{\gamma}{r}$ to denote the (closed) ball of radius $r \in \Rpos$ around $\gamma \in \G$. For $S \subseteq \G$, let $\Ball{S}{r} \defeq \bigcup_{\gamma \in S} \Ball{\gamma}{r}$. For $S$, $T \subseteq \G$, define
	\[
		\dist(S, T) \defeq \inf\set{\dist(\gamma, \delta) \,:\, \gamma \in S,\, \delta \in T}.
	\]
	
	All actions of $\G$ considered here are left actions. Given a set $A$, we equip $A^\G$ with the \emph{shift action} $\sigma_A \colon \G \acts A^\G$, defined by setting, for all $\omega \in A^\G$ and $\gamma \in \G$,
	$$
		(\gamma \cdot \omega)(\delta) \defeq \omega(\delta \gamma) \text{ for all }\delta \in \G.
	$$
	Similarly, $\G$ acts on $\finset{\G}$ and $\finfun{\G}{A}$ by setting, for all $S \in \finset{\G}$, $\phi \in \finfun{\G}{A}$, and $\gamma \in \G$,
	\[
		\gamma \cdot S \defeq \set{\delta\gamma^{-1} \,:\, \delta \in S};
	\]
	\[
	\dom(\gamma \cdot \phi) \defeq \gamma \cdot \dom(\phi)\;\;\;\;\;\;\; \text{and} \;\;\;\;\;\;\; (\gamma \cdot \phi)(\delta) \defeq \phi(\delta \gamma) \text{ for all }\delta \in \dom(\gamma \cdot \phi).
	\]
	
	Note that if $\alpha \colon \G \acts X$ is a continuous action of $\G$ on a Polish space $X$ and $A \subseteq X$ is comeager, then there is a further comeager subset $A' \subseteq A$ that is $\alpha$-invariant, namely $A' \defeq \bigcap_{\gamma \in \G} (\gamma \cdot A)$.
	
	
	To discourse about colorings we need to fix a set of ``colors''; for concreteness, we will use the discrete space $\N$ in that role (although sometimes it might be more convenient to use a different countable discrete space instead; for instance, we use $\N \times \N$ in the proof of Lemma~\ref{lemma:red_to_loc}). By a~\emph{coloring} of a set $S$ we simply mean a map $\omega \colon S \to \N$. A combinatorial coloring problem over $\G$ is meant to specify which colorings of $\G$ are considered ``nice'' or ``acceptable.'' We identify such coloring problems with \emph{subshifts}:
	
	\begin{defn}\label{defn:subshift}
		A \emph{subshift} is a subset of $\N^\G$ that is closed (in the product topology) and invariant under the shift action. The set of all subshifts is denoted by $\S_0(\G, \N)$, and the set of all nonempty subshifts is denoted by $\S(\G, \N)$. 
	\end{defn}
	
	Let $\alpha \colon \G \acts X$ be an action of $\G$ on a set $X$. Each coloring of $X$ then gives rise to a family of colorings of $\G$ parameterized by the elements of $X$. Specifically, given $f \colon X \to \N$ and $x \in X$, we define $\pi_f(x) \colon \G \to \N$ by
	\[
		\pi_f(x)(\gamma) \defeq f(\gamma \cdot x).
	\]
	It is clear that the map $\pi_f \colon X \to \N^\G$ is equivariant, i.e.,
	\[
		\gamma \cdot \pi_f(x) = \pi_f(\gamma \cdot x)
	\]
	for all $x \in X$ and $\gamma \in \G$. Conversely, for each equivariant function $\pi \colon X \to \N^\G$, there is a unique coloring $f \colon X \to \N$ such that $\pi = \pi_f$, namely the one given by $f(x) \defeq \pi(x)(\mathbf{1})$ for all $x \in X$. The map~$\pi_f$ is called the \emph{symbolic representation}, or the \emph{coding map}, for the dynamical system $(X, \G, \alpha, f)$. 
	
	The following definition identifies our main objects of study:
	
	\begin{defn}\label{defn:BM_col}
		Let $\alpha \colon \G \acts X$ be a continuous action of $\G$ on a Polish space~$X$. Given a subshift $\subshift \in \S_0(\G, \N)$, a \emph{Baire measurable $\subshift$-coloring} of $\alpha$ (or of $X$, if $\alpha$ is clear from the context) is a Baire measurable function $f \colon X \to \N$ such that the preimage of $\subshift$ under $\pi_f$ is comeager. The set of all $\subshift \in \S_0(\G, \N)$ such that $\alpha$ admits a Baire measurable $\subshift$-coloring is denoted by $\S_{\BM}(\alpha, \N)$.
	\end{defn}
	
	\begin{remk}
		Clearly, $\S_{\BM}(\alpha, \N) \subseteq \S(\G, \N)$, unless $X = \0$.
	\end{remk}
	
	\begin{remk}
	In view of the bijective correspondence $f \longleftrightarrow \pi_f$ between colorings and equivariant functions, Definition~\ref{defn:BM_col} can be equivalently restated in purely dynamical terms as follows:
	\begin{leftbar}
			\noindent A~continuous action $\alpha \colon \G \acts X$ admits a Baire measurable $\subshift$\=/coloring if and only if there exists a Baire measurable map $\pi \colon X \to \subshift$ which is equivariant on a comeager set.
	\end{leftbar}
	\end{remk}
	
	\subsection{Example: proper graph colorings}
	
	In their seminal paper~\cite{KST99}, Kechris, Solecki, and Todorcevic initiated the study of colorings of graphs that satisfy additional definability constraints, leading to the creation of the field of descriptive combinatorics. Even though our framework concerns groups and group actions rather than graphs, there is a standard way of associating a graph to a (finitely generated) group and to each of its free actions. Namely, assume that $\G$ is generated by a finite symmetric subset $S$ with $\mathbf{1} \not \in S$. The corresponding \emph{Cayley graph} $\operatorname{Cay}(\G, S)$ is the graph with vertex set $\G$ and edge set
	$$
	\set{(\gamma, \delta \gamma) \,:\, \gamma \in \G \text{ and } \delta \in S}.
	$$
	Similarly, for a free\footnote{Recall that an action $\alpha \colon \G \acts X$ is \emph{free} if for all $x \in X$ and $\gamma \in \G$, $\gamma \cdot x = x$ implies $\gamma = \mathbf{1}$.} continuous action $\alpha \colon \G \acts X$ on a Polish space $X$, let $\mathcal{G}(\alpha, S)$ denote the graph \emph{induced} by $\alpha$, i.e., the graph with vertex set $X$ and edge set
	$$
	\set{(x, \delta \cdot x) \,:\, x \in X \text{ and } \delta \in S}.
	$$
	Since $\alpha$ is free, every connected component of $\mathcal{G}(\alpha, S)$ is isomorphic to $\operatorname{Cay}(\G, S)$; specifically, for $x \in X$, the map $\gamma \mapsto \gamma \cdot x$ is an isomorphism from $\operatorname{Cay}(\G, S)$ onto the connected component of $\mathcal{G}(\alpha, S)$ containing $x$ (which coincides with the $\alpha$-orbit of~$x$).
	
	For $k \in \N$, let $\PrCol(k,S)$ denote the set of all \emph{proper $k$-colorings} of $\operatorname{Cay}(\G, S)$, i.e., all functions $\omega \colon \G \to k$ such that $\omega(\gamma) \neq \omega(\delta\gamma)$ whenever $\gamma \in \G$ and $\delta \in S$. It is clear that $\PrCol(k, S)$ is a subshift. The smallest $k$ such that $\PrCol(k,S) \neq \0$ is called the \emph{chromatic number} of $\operatorname{Cay}(\G, S)$ and is denoted by $\chi(\operatorname{Cay}(\G, S))$, or simply $\chi(\G, S)$.
	
	Since every vertex in $\operatorname{Cay}(\G, S)$ has exactly $|S|$ neighbors, it is immediate that $$\chi(\G, S) \leq |S|+1.$$ A cornerstone result in graph theory, so-called Brooks's theorem~\cite[Theorem~5.2.4]{Die00}, implies that, in fact, $$\chi(\G, S) \leq |S| \;\;\;\text{for}\;\;\; |S| \geq 2.$$
	
	For a free continuous action $\alpha \colon \G \acts X$ on a Polish space $X$, the smallest $k$ such that $\PrCol(k,S) \in \S_{\BM}(\alpha, \N)$ is called the \emph{Baire measurable chromatic number} of $\mathcal{G}(\alpha, S)$ and is denoted by $\chi_{\BM}(\mathcal{G}(\alpha, S))$, or simply $\chi_{\BM}(\alpha, S)$. Clearly, $$\chi_{\BM}(\alpha, S) \geq \chi(\G, S)\;\;\;\text{for}\;\;\; X \neq \0.$$
	A somewhat surprising result of Conley and Miller~\cite[Theorem~B]{CM16} implies that $\chi_{\BM}(\alpha, S)$ is also \emph{upper} bounded by a function of $\chi(\G, S)$; more precisely,
	\begin{equation}\label{eq:CM}
		\chi_{\BM}(\alpha, S) \leq 2 \chi(\G, S) - 1.
	\end{equation}
	Another important result concerning Baire measurable chromatic numbers is a (Baire) measurable version of Brooks's theorem due to Conley, Marks, and Tucker-Drob~\cite[Theorem~1.2(2)]{CMT-D16}, which implies that, similarly to the situation with ordinary chromatic numbers, $$\chi_{\BM}(\alpha, S) \leq |S| \;\;\;\text{for}\;\;\; |S| \geq 2.$$
	
	\subsection{A completeness result}\label{subsec:complete}
	
	The aim of this article is to make progress towards the understanding of the structure of the sets $\S_\BM(\alpha, \N)$. The first natural question to ask is, how complex, in descriptive set-theoretic terms, is $\S_\BM(\alpha, \N)$, as a subset of $\S_0(\G, \N)$?
	
	First, we have to make $\S_0(\G, \N)$ a Polish or, at least, a standard Borel space. It is straightforward to check that $\S_0(\G, \N)$ is a Borel subset of the \emph{Effros standard Borel space} $\mathcal{F}(\N^\G)$ and as such is itself standard Borel (for more details on the Effros space see \cite[\secsign{}12.C]{Kec95} and \cite[\secsign{}13.D]{Tse16}). Furthermore, in Section~\ref{sec:ideals} we put a natural Polish topology on $\S_0(\G, \N)$ (which results in the same Borel $\sigma$\=/algebra).
	
	Let $\alpha \colon \G \acts X$ be a free continuous action of $\G$ on a nonempty Polish space $X$. We say that $\alpha$ is \emph{generically smooth} if there is a Baire measurable map $f \colon X \to \R$ such that for all $x$, $y \in X$,
	\[
		 f(x) = f(y) \,\Longleftrightarrow\, y = \gamma \cdot x \text{ for some } \gamma \in \G.
	\]
	For smooth actions, descriptive and finite combinatorics essentially coincide; in particular, it is easy to show that if $\alpha$ is generically smooth, then $\S_{\BM}(\alpha, \N) = \S(\G, \N)$ (see Lemma~\ref{lemma:smooth}). In other words, from the point of view of descriptive combinatorics, smooth actions are trivial and it is only interesting to consider non-smooth ones.
	
	We show that in the non-smooth case, $\S_\BM(\alpha, \N)$ is a complete analytic subset of $\S_0(\G, \N)$; in particular, it is not Borel. Informally, this means that there is no hope for an ``explicit'' description of the subshifts $\subshift$ for which a given non-smooth action $\alpha$ admits a Baire measurable $\subshift$-coloring.
	
	\begin{theo}\label{theo:complete}
		Let $\alpha$ be a free continuous action of $\G$ on a nonempty Polish space. Then
		\begin{itemize}
			\item[--] either $\alpha$ is generically smooth, in which case $\S_{\BM}(\alpha, \N) = \S(\G, \N)$;
			\item[--] or else, the set $\S_{\BM}(\alpha, \N)$ is complete analytic.
		\end{itemize}
	\end{theo}
	
	We prove Theorem~\ref{theo:complete} in Section~\ref{sec:complete}. En route to proving Theorem~\ref{theo:complete}, we show that the set of all Baire measurable maps between two Polish spaces, taken modulo the equivalence relation of equality on a comeager set, can be naturally turned into a standard Borel space (see \secsign\ref{subsec:BM_space}); this construction appears to be new and of independent interest.
	
	\subsection{A combinatorial characterization of $\S_{\BM}(\sigma, \N)$}
	
	The following result was first established by Keane for the $2$- and the $3$-shift and subsequently generalized by Weiss~\cite{Wei00}:
	
	\begin{theo}[{Keane--Weiss~\cite[Theorem~2]{Wei00}}]\label{theo:gen_iso}
		Let $X$, $Y$ be Polish spaces of cardinality at least $2$. Then the shift actions $\sigma_X$, $\sigma_Y$ are generically isomorphic; i.e., there exist comeager shift-invariant subsets $X' \subseteq X^\G$, $Y' \subseteq Y^\G$ with an equivariant homeomorphism $\pi \colon X' \to Y'$ between them.
	\end{theo}
	
	\begin{remk}
		Theorem~2 in~\cite{Wei00} is stated for zero-dimensional spaces only. The result for general Polish spaces follows since every Polish space contains a comeager zero-dimensional subspace.
	\end{remk}
	
	Theorem~\ref{theo:gen_iso} allows us to refer, when meager sets may be ignored, to \emph{the} shift action $\sigma$, meaning \emph{any} shift action $\sigma_X$ for a Polish space $X$ of cardinality at least $2$. Note that this is in striking contrast to the situation in measurable dynamics.
	
	We associate with each subshift a certain countable object, which we call a \emph{$\G$-ideal}.
	
	\begin{defn}
		A subset $\P \subseteq \finfun{\G}{\N}$ is called a \emph{$\G$-ideal} if it is invariant under the action of $\G$ on $\finfun{\G}{\N}$ and closed under restrictions (i.e., if $\phi \subseteq \phi' \in \P$, then $\phi \in \P$).
	\end{defn}
	
	\begin{defn}
		For a subshift $\subshift \in \S_0(\G, \N)$, a map $\phi \in \finfun{\G}{\N}$ is called a \emph{finite $\subshift$\=/coloring} if there exists a coloring $\omega \in \subshift$ extending $\phi$. The set of all finite $\subshift$-colorings is denoted by $\FPC{\subshift}$.
	\end{defn}
	
	Clearly, for any $\subshift\in \S_0(\G, \N)$, the set $\FPC{\subshift}$ is a $\G$-ideal. However, not every $\G$-ideal arises in this way. We call the $\G$-ideals of the form $\FPC{\subshift}$ \emph{extendable} and characterize them combinatorially in Section~\ref{sec:ideals}. There we also assemble a ``dictionary'' of some correspondences between subshifts and extendable $\G$-ideals. They are useful, for example, in defining the topology on $\S_0(\G, \N)$.
	
	The second main result of this article is a purely combinatorial description of the set $\S_\BM(\sigma, \N)$. 
	Roughly speaking, we show that determining whether there exists a Baire measurable $\subshift$-coloring of $\sigma$ is equivalent to settling a question of the following form:
	\begin{equation}\label{eq:ques} \parbox{0.75\textwidth}{``Is it possible to decide whether a given partial coloring $\phi \in \finfun{\G}{\N}$ belongs to $\FPC{\subshift}$ only using `local' information?''}\tag{$\ast$}\end{equation}
	\noindent This question is rather natural, and some of its versions have already been studied in finite combinatorics with no connection to descriptive set theory. One particular interpretation of~\eqref{eq:ques}, which is of special interest in graph theory, is the problem of jointly extending given pre-colorings of substructures that are sufficiently far apart from each other. There is an extensive literature on this subject; see \cite{Alb98, AKW05, DLMP17, PT16} for a small sample. We formalize this idea in Definition~\ref{defn:PEP} as the \emph{join property} of subshifts. Definition~\ref{defn:local} isolates the class of \emph{local} subshifts; locality is stronger than the join property (see Remark after Definition~\ref{defn:local}).
	
	
	Let $\P \subseteq \finfun{\G}{\N}$ be a $\G$-ideal. A function $R \colon \P \to \Rpos$ is \emph{invariant} if $R(\gamma \cdot \phi) = R(\phi)$ for all $\phi \in \P$ and $\gamma \in \G$. We say that $\phi$,~$\psi \in \P$ are \emph{$R$-separated} if
	\begin{equation*}\label{eq:PEP}
		\dist(\dom(\phi), \dom(\psi)) \,>\, R(\phi) + R(\psi).
	\end{equation*}
	
	\begin{defn}\label{defn:PEP}
		Let $\P \subseteq \finfun{\G}{\N}$ be a $\G$-ideal. We say that $\P$ has the \emph{join property} if there is an invariant function $R \colon \P \to \Rpos$ such that for all $k \in \N$, if $\phi_1$, \ldots, $\phi_k \in \P$ are pairwise $R$\=/separated, then
		$\phi_1 \cup \ldots \cup \phi_k \in \P$. A subshift $\subshift \in \S_0(\G, \N)$ has the join property if so does the $\G$-ideal~$\FPC{\subshift}$.
	\end{defn}
	
	\begin{remk}
		For $k = 0$, we interpret the above definition to mean that $\0 \in \P$; in other words, a $\G$-ideal with the join property is necessarily nonempty.
	\end{remk}
	
	
	Given $\phi \in \finfun{\G}{\N}$, an element $\gamma \in \G$, and a radius $r \in \Rpos$, define
	\[
		\phi[\gamma, r] \defeq \rest{\phi}{(\dom(\phi) \cap \Ball{\gamma}{r})}.
	\]
	Let $\P \subseteq \finfun{\G}{\N}$ be a $\G$-ideal. Given a function $r \colon \N \to \Rpos$, we say that $\phi \in \finfun{\G}{\N}$ is \emph{$r$-locally in~$\P$} if for each $\gamma \in \dom(\phi)$,
	$$
	\phi[\gamma, r(\phi(\gamma))] \in \P.
	$$
	The set of all $\phi \in \finfun{\G}{\N}$ that are $r$-locally in~$\P$ is denoted by~$\loc{\P}{r}$.
	Note that since $\P$ is closed under restrictions, we have $\loc{\P}{r} \supseteq \P$ for all $r \colon \N \to \Rpos$.
	
	\begin{defn}\label{defn:local}
		Let $\P \subseteq \finfun{\G}{\N}$ be a $\G$-ideal. We say that $\P$ is \emph{local} if $\P = \loc{\P}{r}$ for some function $r \colon \N \to \Rpos$. A subshift $\subshift \in \S_0(\G, \N)$ is local if so is the $\G$-ideal $\FPC{\subshift}$.
	\end{defn}
	
	\begin{remk}
		Notice that every local $\G$-ideal has the join property. Indeed, suppose that $\P \subseteq \finfun{\G}{\N}$ is a local $\G$-ideal and let $r \colon \N \to \Rpos$ be a function such that $\P = \loc{\P}{r}$. Define an invariant map $R \colon \P \to \Rpos$ by
		\[
			R(\phi) \defeq \sup\set{r(\phi(\gamma)) \,:\, \gamma \in \dom(\phi)}.
		\]
		Let $\phi_1$, \ldots, $\phi_k \in \P$ be pairwise $R$-separated and set $\phi \defeq \phi_1 \cup \ldots \cup \phi_k$. Consider an arbitrary element $\gamma \in \dom(\phi)$. Then $\gamma \in \dom(\phi_i)$ for a unique $1 \leq i \leq k$. Since $\phi_1$, \ldots, $\phi_k$ are pairwise $R$-separated, for each $j \neq i$, we have
		\[
			\Ball{\gamma}{R(\phi_i)} \,\cap\, \dom(\phi_j) \,=\, \0.
		\]
		Since
		$r(\phi(\gamma)) = r(\phi_i(\gamma)) \leq R(\phi_i)$,
		we conclude that
		\[
			\phi[\gamma, r(\phi(\gamma))] \,\subseteq \, \phi[\gamma, R(\phi_i)] \,=\, \phi_i[\gamma, R(\phi_i)] \,\subseteq\, \phi_i \,\in\, \P.
		\]
		Therefore, $\phi \in \loc{\P}{r}$. As $\loc{\P}{r} = \P$, we obtain $\phi \in \P$, as desired.
	\end{remk}
	
	We need one last definition:
	
	\begin{defn}\label{defn:red}
		If $\subshift$, $\subshift' \in \S_0(\G, \N)$ are subshifts, then $\subshift$ is \emph{reducible} to $\subshift'$, in symbols $\subshift \succeq \subshift'$, if there is a map $\rho \colon \N \to \N$, called a \emph{reduction}, such that for all $\omega \in \subshift'$, we have $\rho \circ \omega \in \subshift$.
	\end{defn}
	
	\begin{remk}
		A special case of reducibility is when $\subshift \supseteq \subshift'$. Indeed, if $\subshift \supseteq \subshift'$, then the identity map $\operatorname{id}_\N \colon \N \to \N$ is a reduction from $\subshift$ to $\subshift'$. This explains the orientation of the symbol ``$\succeq$.''
	\end{remk}
	
	\begin{remk}
		If $\subshift \succeq \subshift'$ and $\subshift' \in \S_\BM(\alpha, \N)$ for some continuous action $\alpha \colon \G \acts X$ on a Polish space $X$, then $\subshift \in \S_\BM(\alpha, \N)$ as well. Indeed, if $\rho \colon \N\to \N$ is a reduction from $\subshift$ to $\subshift'$ and $f \colon X \to \N$ is a Baire measurable $\subshift'$-coloring of $\alpha$, then $\rho \circ f$ is a Baire measurable $\subshift$-coloring of $\alpha$.
	\end{remk}
	
	
	Finally, we are ready to state our result:
	
	\begin{theo}\label{theo:combi}
		The following statements are equivalent for a subshift $\subshift \in \S_0(\G, \N)$:
		\begin{enumerate}[label=\normalfont{(\roman*)}]
			\item\label{item:combi:BM} $\subshift \in \S_\BM(\sigma, \N)$;
			\item\label{item:combi:amalgamation} $\subshift \supseteq \subshift'$ for some subshift $\subshift'$ with the join property;
			\item\label{item:combi:local} $\subshift \succeq \subshift'$ for some local subshift $\subshift'$.
		\end{enumerate}
	\end{theo}
	
	We prove Theorem~\ref{theo:combi} in Section~\ref{sec:combi}.
	
	
	\subsection{Some corollaries}\label{subsec:corls}
	
	As mentioned previously, the join property and its analogs have been an object of study in graph theory (although Definition~\ref{defn:PEP} does not appear to have been explicitly articulated before). In particular, implication \ref{item:combi:amalgamation} $\Longrightarrow$ \ref{item:combi:BM} of Theorem~\ref{theo:combi} can be used to derive bounds on Baire measurable chromatic numbers from known results in finite combinatorics. For instance, deep results of Postle and Thomas~\cite{PT16} yield the following:
	
	\begin{corl}\label{corl:planar}
		Suppose that $\G$ is generated by a finite symmetric set $S \subset \G$ with $\mathbf{1} \not \in S$ such that the corresponding Cayley graph $\mathcal{G} \defeq \operatorname{Cay}(\G, S)$ is planar. Then
		\begin{equation}\label{eq:planar}
			\chi_\BM(\sigma, S) \leq \begin{cases}
			3 &\text{if $\mathcal{G}$ contains no cycles of lengths $3$ and $4$};\\
			4 &\text{if $\mathcal{G}$ contains a cycle of length $4$ but not of length $3$};\\
			5 &\text{otherwise}.
			\end{cases}
		\end{equation}
	\end{corl}
	\begin{proof}
		Assume that $\G$ and $S$ satisfy the above assumptions and let $k$ denote the quantity on the right hand side of~\eqref{eq:planar}. The fact that $\PrCol(S, k)$ is a subshift with the join property is a consequence of~\cite[Theorem~8.10]{PT16}.
	\end{proof}
	
	Note that the best upper bounds for $\chi_\BM(\sigma, S)$ under the assumptions of Corollary~\ref{corl:planar} that follow from previously known results are $\chi_\BM(\sigma, S) \leq 7$ in general and $\chi_\BM(\sigma, S) \leq 5$ if $\operatorname{Cay}(\G, S)$ contains no cycles of length $3$; these follow from combining~\cite[Theorem~B]{CM16} (see~\eqref{eq:CM} above) with the Four Color Theorem~\cite[Theorem~5.1.1]{Die00} and Gr\"otzsch's theorem~\cite[Theorem~5.1.3]{Die00} respectively. The proof of \cite[Theorem~8.10]{PT16} due to Postle and Thomas is quite difficult.
	
	Locality of a subshift is often rather easy to check, which makes condition~\ref{item:combi:local} of Theorem~\ref{theo:combi} a convenient tool for constructing subshifts in $\S_\BM(\sigma, \N)$ with additional interesting properties. To illustrate this, in~\secsign\ref{subsec:not_universal} we prove the following:
	
	\begin{corl}\label{corl:not_universal}
		There exists a free continuous action $\alpha$ of $\G$ on a Polish space such that \[\S_\BM(\alpha, \N) \centernot \supseteq \S_\BM(\sigma, \N).\]
	\end{corl}
	
	We find Corollary~\ref{corl:not_universal} somewhat surprising. Indeed, due to Theorem~\ref{theo:gen_iso}, all non-trivial shift actions of $\G$ admit exactly the same types of Baire measurable colorings. Analogous statements hold in the purely Borel context and in the context of approximate measure colorings; the former follows from a result of Seward and Tucker-Drob~\cite[Theorem~1.1]{ST-D16}, the latter---from the Ab\'ert--Weiss theorem on weak containment of Bernoulli shifts~\cite[Theorem~1]{AW13}. However, both in the Borel and in the approximate measure frameworks, the shift actions are actually the \emph{hardest} ones to color (which also follows from \cite[Theorem~1.1]{ST-D16} and \cite[Theorem~1]{AW13}), whereas, as Corollary~\ref{corl:not_universal} asserts, that is \emph{not} the case in the Baire category setting.
	
	\section{Extendable $\G$-ideals}\label{sec:ideals}
	
	\noindent Due to their combinatorial nature, we sometimes find working with $\G$-ideals more convenient than referring to subshifts directly. In this section we summarize some useful correspondences between the two kinds of objects. Most statements made here follow readily from definitions.
	
	Given a $\G$-ideal $\P \subseteq \finfun{\G}{\N}$, an \emph{$\P$-coloring} is a map $\omega \colon \G \to \N$ such that
	\[
		\rest{\omega}{S} \in \P\;\;\; \text{for all} \;\;\;S \in \finset{\G}.
	\]
	The set of all $\P$-colorings is denoted $\Sol(\P)$. It is clear that $\Sol(\P) \subseteq \N^\G$ is a subshift.
	
	\begin{defn}\label{defn:ext}
		A $\G$-ideal $\P \subseteq \finfun{\G}{\N}$ is \emph{extendable} if for every $\phi \in \P$ and $\gamma \in \G \setminus \dom(\phi)$, there is a color $c \in \N$ such that $\phi \cup \set{(\gamma, c)} \in \P$.
	\end{defn}
	
	\begin{prop}\label{prop:finite}
		Let $\P \subseteq \finfun{\G}{\N}$ be a $\G$-ideal. The following statements are equivalent:
		\begin{itemize}
			\item[--] $\P = \FPC{\subshift}$ for some $\subshift \in \S_0(\G, \N)$;
			\item[--] $\P$ is extendable.
		\end{itemize}
		If $\P$ is extendable, then the subshift $\subshift \in \S_0(\G, \N)$ such that $\P = \FPC{\subshift}$ is unique, namely $\subshift = \Sol(\P)$.
	\end{prop}
	
	The proof of Proposition~\ref{prop:finite} is straightforward, and we do not spell it out here.
	
	Note that the set $\mathbf{Ext}(\G, \N)$ of all extendable $\G$-ideals is a $G_\delta$ subset of the power set of $\finfun{\G}{\N}$. 
	By Alexandrov's theorem~\cite[Theorem~3.11]{Kec95}, $\mathbf{Ext}(\G, \N)$ is Polish in its relative topology. The bijection between $\mathbf{Ext}(\G, \N)$ and $\S_0(\G, \N)$, given by the maps $\Sol \colon \mathbf{Ext}(\G, \N) \to \S_0(\G, \N)$ and $\mathbf{Fin} \colon \S_0(\G, \N) \to \mathbf{Ext}(\G, \N)$, allows us to transfer the Polish topology from $\mathbf{Ext}(\G, \N)$ to $\S_0(\G, \N)$, thus turning $\S_0(\G, \N)$ into a Polish space. Explicitly, the Polish topology on $\S_0(\G, \N)$ is generated by the open sets of the form
	\[
		\set{\subshift \in \S_0(\G, \N) \,:\, \phi \in \FPC{\subshift}} \;\;\;\text{and}\;\;\;\set{\subshift \in \S_0(\G, \N) \,:\, \phi \not\in \FPC{\subshift}},
	\]
	where $\phi$ is ranging over $\finfun{\G}{\N}$.
	
	The next definition is the analog of Definition~\ref{defn:red} for $\G$-ideals:
	
	\begin{defn}
		If $\P$, $\P' \subseteq \finfun{\G}{\N}$ are $\G$-ideals, then $\P$ is \emph{reducible} to $\P'$, in symbols $\P \succeq \P'$, if there is a map $\rho \colon \N \to \N$, called a \emph{reduction}, such that for all $\phi \in \P'$, we have $\rho \circ \phi \in \P$.
	\end{defn}
	
	The following statements are also straightforward:
	
	\begin{prop}
		Let $\subshift$, $\subshift' \in S_0(\G, \N)$. Then
		\begin{itemize}
			\item[--] $\subshift \supseteq \subshift'$ if and only if $\FPC{\subshift} \supseteq \FPC{\subshift'}$; and
			\item[--] $\subshift \succeq \subshift'$ if and only if $\FPC{\subshift} \succeq \FPC{\subshift'}$.
		\end{itemize}
	\end{prop}
	
	Finally, given a $\G$-ideal $\P$ and a continuous action $\alpha \colon \G \acts X$ on a Polish space $X$, a \emph{Baire measurable $\P$-coloring} of $\alpha$ is the same as a Baire measurable $\Sol(\P)$\=/coloring of $\alpha$.
	
	\section{Proof of Theorem~\ref{theo:complete}}\label{sec:complete}
	
	\subsection{The space of Baire measurable functions}\label{subsec:BM_space}
	
	For the rest of this subsection (save Corollary~\ref{corl:analytic}), we fix a Polish space~$X$ and a standard Borel space $Y$. Two Baire measurable functions $f$, $g \colon X \to Y$ are \emph{equal on a comeager set}, or \emph{$\ast$-equal}, in symbols $f \stareq g$, if the set $\set{x \in X \,:\, f(x) = g(x)}$ is comeager. The set of all Baire measurable functions from $X$ to $Y$, taken modulo the equivalence relation of $\ast$-equality, is denoted by $\forceset{X}{Y}$. For a nonempty open set $U \subseteq X$ and a Borel subset $A \subseteq Y$, let
	\[
		\forceset{U}{A} \defeq \set{f \in \forceset{X}{Y} \,:\, U \Vdash f^{-1}(A)}.
	\]
	Let $\Bairealg{X}{Y}$ denote the $\sigma$-algebra on $\forceset{X}{Y}$ generated by the sets of the form $\forceset{U}{A}$ for all nonempty open $U \subseteq X$ and Borel $A \subseteq Y$.
	
	\begin{theo}\label{theo:st_Borel}
		The measurable space $(\forceset{X}{Y}, \Bairealg{X}{Y})$ is standard Borel.
	\end{theo}
	
	A $\sigma$-algebra $\mathfrak{S}$ on a set $Z$ \emph{separates points} if for all $z$, $z' \in Z$, if $z \neq z'$, then there exists $A \in \mathfrak{S}$ such that $z \in A$ and $z' \not \in A$.
	
	\begin{lemma}\label{lemma:seppoints}
		The $\sigma$-algebra $\Bairealg{X}{Y}$ separates points.
	\end{lemma}
	\begin{proof}
		Suppose that $f$, $g \in \forceset{X}{Y}$ are not $\ast$-equal, i.e., the set $\set{x \in X \,:\, f(x) \neq g(x)}$ is nonmeager. Fix an arbitrary Polish topology on $Y$ that generates its Borel $\sigma$-algebra. By~\cite[Theorem~8.38]{Kec95}, there is a comeager subset $X' \subseteq X$ such that the restricted functions $\rest{f}{X'}$, $\rest{g}{X'}$ are continuous. Then the set $\set{x \in X' \,:\, f(x) \neq g(x)}$ is also nonmeager, and hence nonempty. Consider any $x_0 \in X'$ with $f(x_0) \neq g(x_0)$ and let $V$, $W \subset Y$ be disjoint open neighborhoods of $f(x_0)$ and $g(x_0)$ respectively. By the continuity of $\rest{f}{X'}$ and $\rest{g}{X'}$, there exists an open neighborhood $U \subseteq X$ of $x_0$ such that $f(x) \in V$ and $g(x) \in W$ for all $x \in U \cap X'$. This yields $f \in \forceset{U}{V}$ and $g \in \forceset{U}{W}$. As $\forceset{U}{V} \cap \forceset{U}{W} = \0$, and so $g \not \in \forceset{U}{V}$, the proof is complete.
	\end{proof}
	
	\begin{lemma}\label{lemma:generators}
		Let $\mathcal{U}$ be a countable basis for the topology on $X$ consisting of nonempty open sets, and let $\mathcal{A}$ be a countable basis for the Borel $\sigma$-algebra on $Y$. Then $\Bairealg{X}{Y}$ is generated by the family of sets
		\[
			\forceset{\mathcal{U}}{\mathcal{A}} \defeq \set{\forceset{U}{A} \,:\, U \in \mathcal{U},\, A \in \mathcal{A}}.
		\]
	\end{lemma}
	\begin{proof}
		Since for $U_0$, $U_1$, $\ldots \in \mathcal{U}$ and $A \in \mathcal{A}$, we have
		\[
			\forceset{\bigcup_{i=0}^\infty U_i}{A} \,=\, \bigcap_{i=0}^\infty \forceset{U_i}{A},
		\]
		the $\sigma$-algebra $\Bairealg{X}{Y}$ is generated by the sets $\forceset{U}{A}$ with $U \in \mathcal{U}$. Since we also have
		\[
			\forceset{U}{A^\mathsf{c}} \,=\, \textstyle\bigcap \set{\forceset{V}{A}^\mathsf{c} \,:\, V \in \mathcal{U}, \, V \subseteq U},
		\]
		where $(\cdot)^\mathsf{c}$ denotes set complement, and
		\[
			\forceset{U}{\bigcap_{i=0}^\infty A_i} \, = \, \bigcap_{i=0}^\infty \forceset{U}{A_i},
		\]
		we conclude that $\Bairealg{X}{Y}$ is indeed generated by the sets $\forceset{U}{A}$ with $U \in \mathcal{U}$ and $A \in \mathcal{A}$.
	\end{proof}
	
	
	\begin{prop}\label{prop:criterion}
		Let $(Z, \mathfrak{S})$ be a measurable space such that the $\sigma$-algebra $\mathfrak{S}$ separates points. Let $\mathcal{A} \subseteq \mathfrak{S}$ be a countable generating set for $\mathfrak{S}$. For each $z \in Z$, define $\theta_z \colon \mathcal{A} \to 2$ as follows:
		$$
		\theta_z (A) \defeq \begin{cases}
		1 &\text{if } z \in A;\\
		0 &\text{if } z \not \in A.
		\end{cases}
		$$
		Let $\Theta$ denote the image of $Z$ under the map $z \mapsto \theta_z$. Then $(Z, \mathfrak{S})$ is standard Borel if and only if $\Theta$ is a Borel subset of the product space $2^{\mathcal{A}}$.
	\end{prop}
	\begin{proof}
		Let $\mathfrak{B} \defeq \mathfrak{B}(2^\mathcal{A})$ denote the Borel $\sigma$-algebra on $2^\mathcal{A}$. Since $\mathfrak{S}$ separates points, the map $z \mapsto \theta_z$ is injective; by construction, it is therefore an isomorphism of measurable spaces $(Z, \mathfrak{S})$ and $(\Theta, \rest{\mathfrak{B}}{\Theta})$, where $\rest{\mathfrak{B}}{\Theta}$ is the relative $\sigma$-algebra on $\Theta$. 
		Thus, $(Z, \mathfrak{S})$ is standard Borel if and only if so is $(\Theta, \rest{\mathfrak{B}}{\Theta})$; by the Luzin--Suslin theorem~\cite[Theorem~15.1]{Kec95}, the latter condition is equivalent to~$\Theta$ being a Borel subset of $2^\mathcal{A}$.
	\end{proof}
	
	\begin{proof}[Proof of Theorem~\ref{theo:st_Borel}]
		Let $\mathcal{U}$ be a countable basis for the topology $X$ consisting of nonempty open sets. Using the Borel isomorphism theorem \cite[Theorem~15.6]{Kec95}, we can choose a zero\=/dimensional compact metrizable topology on $Y$ that generates its Borel $\sigma$-algebra; let $\mathcal{A}$ be a countable basis for that topology consisting of sets that are simultaneously open and closed.
		
		For each $f \in \forceset{X}{Y}$, define $\theta_f \colon \mathcal{U} \times \mathcal{A} \to 2$ as follows:
		\[
		\theta_f(U, A) \defeq \begin{cases}
		1 &\text{if } U \Vdash f^{-1}(A);\\
		0 &\text{if } U \not \Vdash f^{-1}(A),
		\end{cases}
		\]
		and let $\Theta$ denote the image of $\forceset{X}{Y}$ under the map $f \mapsto \theta_f$. In view of Lemmas~\ref{lemma:seppoints},~\ref{lemma:generators}, and Proposition~\ref{prop:criterion}, we only need to show that $\Theta$ is a Borel subset of the product space $2^{\mathcal{U} \times \mathcal{A}}$.
		
		Let $\Theta'$ denote the set of all functions $\theta \colon \mathcal{U} \times \mathcal{A} \to 2$ satisfying the following two requirements:
		\begin{enumerate}[label=\normalfont{(\arabic*)}]
			\item\label{item:fip} for all $k \in \N$, $U_1$, \ldots, $U_k \in \mathcal{U}$, and $A_1$, \ldots, $A_k \in \mathcal{A}$,
			\[
			\begin{array}{lcl}
			\text{if} & &U_1 \cap \ldots \cap U_k \neq \0 \;\;\;\;\text{and}\;\;\;\; \theta(U_1, A_1) = \ldots = \theta(U_k, A_k) = 1,\\
			\text{then} & &A_1 \cap \ldots \cap A_k \neq \0;
			\end{array}
			\]
			\item\label{item:proof} for all $U \in \mathcal{U}$ and $A \in \mathcal{A}$, if $\theta(U, A) = 0$, then there exist $V \in \mathcal{U}$ and $B \in \mathcal{A}$ such that
			\[V \subseteq U, \;\;\;\;\; B \cap A = \0, \;\;\;\;\; \text{and} \;\;\;\;\; \theta(V, B) = 1.\]
		\end{enumerate}
		Note that $\Theta'$ is evidently a Borel (in fact, $G_\delta$) subset of $2^{\mathcal{U} \times \mathcal{A}}$. 
		
		\begin{claim}\label{claim:first_inclusion}
			$\Theta \subseteq \Theta'$.
		\end{claim}
		\begin{claimproof}
			Let $f \in \forceset{X}{Y}$. We need to show that $\theta_f$ satisfies conditions~\ref{item:fip} and~\ref{item:proof}.
			
			\begin{itemize}[wide]
				\item[\ref{item:fip}] If $U_1$, \ldots, $U_k \in \mathcal{U}$ and $A_1$, \ldots, $A_k \in \mathcal{A}$ are such that
				\[
				U_1 \cap \ldots \cap U_k \neq \0 \;\;\;\;\text{and}\;\;\;\; \theta_f(U_1, A_1) = \ldots = \theta_f(U_k, A_k) = 1,
				\]
				then $U_1 \cap \ldots \cap U_k$ is nonempty open and
				\[
				U_1 \cap \ldots \cap U_k \,\Vdash\, f^{-1}(A_1) \cap \ldots \cap f^{-1}(A_k) \,=\, f^{-1}(A_1 \cap \ldots \cap A_k),
				\]
				implying that $f^{-1}(A_1 \cap \ldots \cap A_k)$ is nonmeager, and hence $A_1 \cap \ldots \cap A_k \neq \0$.
				
				\item[\ref{item:proof}] Let $U \in \mathcal{U}$ and $A \in \mathcal{A}$ be such that $\theta_f(U, A) = 0$, i.e., $U \not \Vdash f^{-1}(A)$. The sets in $\mathcal{A}$ are simultaneously open and closed; in particular, the complement of $A$ is open and hence equal to the union of all $B \in \mathcal{A}$ with $B \cap A = \0$. Therefore, for some $B \in \mathcal{A}$ with $B \cap A = \0$, the set $U \cap f^{-1}(B)$ is nonmeager. By the Baire alternative, there is $V \in \mathcal{U}$ such that $V \subseteq U$ and $V \Vdash f^{-1}(B)$, i.e., $\theta_f(V, B) = 1$. \qedhere
			\end{itemize}
		\end{claimproof}
		
		\begin{claim}\label{claim:second_inclusion}
			$\Theta' \subseteq \Theta$.
		\end{claim}
		\begin{claimproof}
			Let $\theta \in \Theta'$. We need to find a function $f \in \forceset{X}{Y}$ such that $\theta_f = \theta$.
			
			For each $x \in X$, let
			\[
			\mathcal{A}_x \defeq \set{A \in \mathcal{A} \,:\, \theta(U, A) = 1 \text{ for some } U \in \mathcal{U} \text{ with } U \ni x}.
			\]
			Note that $\mathcal{A}_x$ is a family of closed subsets of the compact space~$Y$, and condition~\ref{item:fip} implies that it has the finite intersection property; therefore, $R_x \defeq \bigcap \mathcal{A}_x$ is a nonempty compact set. The set
			\[
			R \defeq \set{(x, y) \in X \times Y \,:\, y \in R_x}
			\]
			is Borel (in fact, closed) in $X \times Y$, so by~\cite[Theorem~28.8]{Kec95}, there exists a Borel map $f \colon X \to Y$ such that $f(x) \in R_x$ for all $x \in X$. We claim that $\theta_f = \theta$ for any such $f$.
			
			Indeed, let $U \in \mathcal{U}$ and $A \in \mathcal{A}$. If $\theta(U, A) = 1$, then for all $x \in U$, \[f(x) \in R_x \subseteq A,\] so $U \subseteq f^{-1}(A)$, and thus $\theta_f(U, A) = 1$. On the other hand, if $\theta(U, A) = 0$, then, by~\ref{item:proof}, there exist sets $V \in \mathcal{U}$ and $B \in \mathcal{A}$ such that \[V \subseteq U, \;\;\;\;\; B \cap A = \0, \;\;\;\;\; \text{and} \;\;\;\;\; \theta(V, B) = 1.\] By the previous argument, $\theta_f(V, B) = 1$, i.e., $V \Vdash f^{-1}(B)$. Since $B \cap A = \0$, this implies $U \not \Vdash f^{-1}(A)$, and so $\theta_f(U, A) = 0$.
		\end{claimproof}
		Together, Claims~\ref{claim:first_inclusion} and~\ref{claim:second_inclusion} yield $\Theta = \Theta'$; in particular, $\Theta$ is Borel.
	\end{proof}
	
	\begin{corl}\label{corl:analytic}
		Let $\alpha \colon \G \acts X$ be a continuous action of $\G$ on a Polish space $X$. Then the set $\S_\BM(\alpha, \N)$ is analytic.
	\end{corl}
	\begin{proof}
		It is routine to check that the set
		\[
			\mathbf{Hom}(\alpha, \sigma_\N) \defeq \set{\pi \in \forceset{X}{\N^\G} \,:\, \pi \text{ is equivariant on a comeager set}}
		\]
		is a Borel subset of $\forceset{X}{\N^\G}$; furthermore, the set
		\[
			\set{(\pi, \subshift) \in \mathbf{Hom}(\alpha, \sigma_\N) \times \S_0(\G, \N) \,:\, \pi^{-1}(\subshift) \text{ is comeager}}
		\]
		is a Borel subset of $\mathbf{Hom}(\alpha, \sigma_\N) \times \S_0(\G, \N)$. As
		\[
			\S_\BM(\alpha, \N) = \set{\subshift \in \S_0(\G, \N) \,:\, \pi^{-1}(\subshift) \text{ is comeager for some } \pi \in \mathbf{Hom}(\alpha, \sigma_\N)},
		\]
		we see that $\S_\BM(\alpha, \N)$ is analytic.
	\end{proof}
	
	\subsection{Smoothness}\label{subsec:smooth}
	
	Let $X$ be a Polish (or, more generally, standard Borel) space. An equivalence relation $\E$ on $X$ is said to be Borel if the set
	\[
		\set{(x, y) \in X \times X \,:\, x \text{ and } y \text{ are $\E$-equivalent} }
	\]
	is Borel in $X \times X$. A Borel equivalence relation $\E$ on $X$ is \emph{smooth} if there is a Borel function $f \colon X \to \R$ such that for all $x$, $y \in X$,
	\[
		f(x) = f(y) \,\Longleftrightarrow\, x \text{ and } y \text{ are $\E$-equivalent}.
	\]
	A set $T \subseteq X$ is a \emph{transversal} for $\E$ if every $\E$-class intersects $T$ in exactly one point. The following useful proposition follows from the Luzin--Novikov theorem~\cite[Theorem~18.10]{Kec95}:
	
	\begin{prop}[{\cite[Proposition~20.6]{Tse16}}]\label{prop:smooth}
		Let $\E$ be a Borel equivalence relation on a Polish space $X$. Suppose that every $\E$-class is countable. Then the following statements are equivalent:
		\begin{itemize}
			\item[--] $\E$ is smooth;
			\item[--] there exists a Borel transversal $T \subseteq X$ for $\E$.
		\end{itemize}
	\end{prop}
	
	Given a continuous action $\alpha \colon \G \acts X$ on a Polish space $X$, let $\E_\alpha$ denote the induced \emph{orbit equivalence relation}, i.e., the equivalence relation on $X$ whose classes are precisely the orbits of $\alpha$. Note that $\E_\alpha$ is Borel (in fact, the set $\set{(x, y) \in X \times X \,:\, x \text{ and } y \text{ are $\E_\alpha$-equivalent}}$ is closed). The definition of generic smoothness for actions of $\G$ from~\secsign\ref{subsec:complete} then can be phrased as follows:
	\begin{leftbar}
		\noindent A continuous action $\alpha \colon \G \acts X$ is generically smooth if and only if there exists a comeager $\alpha$\=/invariant Borel subset $X' \subseteq X$ such that the relation $\E_\alpha$ restricted to $X'$ is smooth.
	\end{leftbar}
	
	\begin{lemma}\label{lemma:smooth}
		If $\alpha$ is a generically smooth free continuous action of $\G$ on a nonempty Polish space, then $\S_\BM(\alpha, \N) = \S(\G, \N)$.
	\end{lemma}
	\begin{proof}
		Let $\alpha \colon \G \acts X$ be as in the statement of the lemma. The inclusion $\S_\BM(\alpha, \N) \subseteq \S(\G, \N)$ is clear as $X \neq \0$. To prove the other inclusion, consider any nonempty subshift $\subshift \in \S(\G, \N)$ and let $\omega \in \subshift$ be an arbitrary coloring. After discarding a meager set if necessary, we may assume that~$\E_\alpha$ is smooth; Proposition~\ref{prop:smooth} then gives a Borel transversal $T \subseteq X$ for $\E_\alpha$. Since $\alpha$ is free, for each $x \in X$, there is a unique element $\gamma_x \in \G$ such that $(\gamma_x)^{-1} \cdot x \in T$. Set $f(x) \defeq \omega(\gamma_x)$ for all $x \in X$. Then for all $x \in X$,
		\[
			\pi_f(x) \,=\, \gamma_x \cdot \omega \,\in\, \subshift,
		\]
		i.e., $f$ is a desired Baire measurable $\subshift$-coloring of $\alpha$.
	\end{proof}
	
	The remainder of this section is dedicated to proving the ``hard'' part of Theorem~\ref{theo:complete}: the completeness of the set $\S_\BM(\alpha, \N)$ for generically non-smooth $\alpha$.
	
		Let $\subshift \in \S_0(\G, \N)$ be a subshift. We say that $\subshift$ is \emph{easy} if $\subshift \in \S_\BM(\alpha, \N)$ for every free continuous action $\alpha$ of $\G$ on a Polish space; we say that $\subshift$ is \emph{hard} if $\subshift \in \S_\BM(\alpha, \N)$ only for generically smooth~$\alpha$. 
		
		\begin{lemma}\label{lemma:hard}
			Let $\subshift \in \S_0(\G, \N)$ be a subshift. Suppose that for each $\omega \in \subshift$, there is some $c \in \N$ such that the set $\set{\gamma \in \G \,:\, \omega(\gamma) = c}$ contains precisely one element. Then $\subshift$ is hard.
		\end{lemma}
		\begin{proof}
			For each $\omega \in \subshift$, let
			\[
			c_\omega \defeq \min \set{c \in \N \,:\, |\set{\gamma \in \G \,:\, \omega(\gamma) = c}| = 1}.
			\]
			Define a Borel set $T \subseteq \subshift$ by
			\[
			T \defeq \set{\omega \in \subshift \,:\, c_\omega = \omega(\mathbf{1})}.
			\]
			Let $\alpha \colon \G \acts X$ be a continuous action of $\alpha$ on a Polish space $X$ and suppose that $f \colon X \to \N$ is a Baire measurable $\subshift$-coloring of $\alpha$. After passing to a comeager subset if necessary, we may assume that the map $f$ is Borel and $X = (\pi_f)^{-1}(\subshift)$. Then $(\pi_f)^{-1}(T)$ is a Borel transversal for $\E_\alpha$.
		\end{proof}
		
		\begin{lemma}\label{lemma:hard_union}
			Let $\subshift_0$, $\subshift_1$, $\ldots \in \S_0(\G, \N)$ be a countable sequence of hard subshifts. If $\subshift \defeq \bigcup_{i = 0}^\infty \subshift_i$ is a subshift, then $\subshift$ is also hard.
		\end{lemma}
		\begin{proof}
			Let $\alpha \colon \G \acts X$ be a continuous action of $\alpha$ on a Polish space $X$ and suppose that $f \colon X \to \N$ is a Baire measurable $\subshift$-coloring of $\alpha$. Set $X_i \defeq (\pi_f)^{-1}(\subshift_i)$. After discarding a meager subset if necessary, we may assume that the map $f$ is Borel and $X = (\pi_f)^{-1}(\subshift) = \bigcup_{i = 0}^\infty X_i$. Passing to an even further comeager subset, we may assume that the relation $\E_\alpha$ restricted to each $X_i$ is smooth. Using Proposition~\ref{prop:smooth}, we obtain Borel transversals $T_i \subseteq X_i$ for the restricted relations. Let
			\[
				\begin{array}{ll}
				S_0 &\defeq\, T_0;\\
				S_{i+1} &\defeq\, T_{i+1} \setminus \bigcup_{j = 0}^i X_j \;\;\;\text{for all } i \in \N,
				\end{array}
			\]
			and set $S \defeq \bigcup_{i=0}^\infty S_i$. Then $S$ a Borel transversal for $\E_\alpha$.
		\end{proof}
	
	\subsection{Combinatorial lemmas}\label{subsec:com_lem}
	
	In this subsection we describe the main combinatorial construction behind our proof of Theorem~\ref{theo:complete}.
	
	\begin{lemma}\label{lemma:infty}
		Let $(d_0, d_1, \ldots) \in \Rpos^\N$ be a sequence such that a ball of radius $d_0$ in $\G$ contains at least~$2$ elements, and for each $c \in \N$, a ball of radius $d_{c+1}$ in $\G$ contains two disjoint balls of radius $d_c$. Suppose that $\omega \colon \G \to \N$ is a coloring such that for all $c \in \N$, 
		\[
			\inf \set{\dist(\gamma, \delta) \,:\, \gamma, \delta \in \G, \, \gamma \neq \delta, \, \omega(\gamma) = \omega(\delta) = c} \, > \, 2d_c.
		\]
		Then $\omega$ uses infinitely many colors, i.e., the set $\set{\omega(\gamma) \,:\, \gamma \in \G}$ is infinite.
	\end{lemma}
	\begin{proof}
		We use induction on $c$ to show that any ball of radius $d_c$ in $\G$ contains an element $\gamma$ with $\omega(\gamma) > c$. For $c = 0$, the assertion follows from the fact that each ball of radius $d_0$ contains at least~$2$ elements, and it is impossible for both of them to have color $0$, since the distance between any two distinct elements $\gamma$, $\delta$ with $\omega(\gamma) = \omega(\delta) = 0$ is strictly greater than $2d_0$. Now assume that the assertion has been verified for some $c$ and consider any ball of radius $d_{c+1}$. It contains two disjoint balls of radius $d_c$, so it must, by the inductive hypothesis, contain two distinct elements $\gamma$, $\delta$ with $\omega(\gamma)$, $\omega(\delta) > c$. As $\dist(\gamma, \delta) \leq 2d_{c+1}$, it is impossible to have $\omega(\gamma) = \omega(\delta) = c+1$, so at least one of $\omega(\gamma)$, $\omega(\delta)$ exceeds $c+1$. 
	\end{proof}
	
	\begin{remk}
		Beside its application in the proof of Theorem~\ref{theo:complete}, Lemma~\ref{lemma:infty} will be used once more in the proof of Corollary~\ref{corl:not_universal}.
	\end{remk}
	
	Let $\alpha \colon \G \acts X$ be a free action of $\G$. For $x$, $y \in X$, write
	\[
		\dist(x, y) \defeq \begin{cases}
			\dist(\mathbf{1}, \gamma) &\text{if } \gamma \in \G \text{ is such that } \gamma \cdot x = y;\\
			\infty &\text{if } x \text{ and } y \text{ are in different $\alpha$-orbits}.
		\end{cases}
	\]
	Due to the right-invariance of the metric $\dist$, for all $x \in X$ and $\gamma$, $\delta \in \G$, we have
	\[
		\dist(\gamma \cdot x, \delta \cdot x) \,=\, \dist(\gamma, \delta).
	\]
	The next lemma is essentially a restatement of \cite[Lemma~3.1]{MU16}; we include its proof here for completeness.
	
	\begin{lemma}[{ess. Marks--Unger~\cite[Lemma~3.1]{MU16}}]\label{lemma:MU}
		Let $\alpha \colon \G \acts X$ be a free continuous action of~$\G$ on a nonempty Polish space $X$. Then for every sequence $(d_0, d_1, \ldots) \in \Rpos^\N$, there exists a Baire measurable coloring $f \colon X \to \N$ such that for all $c \in \N$,
		\[
			\inf \set{\dist(x, y) \,:\, x, y \in X, \, x \neq y, \, f(x) = f(y) = c} \, > \, d_c.
		\]
	\end{lemma}
	\begin{proof}
		It suffices to show that there exists a partial Baire measurable map $f \colon X \rightharpoonup \N$ defined on a comeager subset of $X$ and such that for all $c \in \N$,
		\begin{equation}\label{eq:far}
			\inf \set{\dist(x, y) \,:\, x, y \in \dom(f), \, x \neq y, \, f(x) = f(y) = c} \, > \, d_c.
		\end{equation}
		
		For $c \in \N$, let $\mathcal{G}_c$ denote the graph with vertex set $X$ and edge set
		\[
			\set{(x, y) \in X \times X \,:\, x \neq y \text{ and } \dist(x,y) \leq d_c}.
		\]
		The graph $\mathcal{G}_c$ is Borel (closed, in fact), and the neighborhood of every vertex in $\mathcal{G}_c$ is finite, so $\mathcal{G}_c$ admits a Borel proper $\N$-coloring~\cite[Proposition~4.5]{KST99}, i.e., a Borel function $\eta_c \colon X \to \N$ such that $\eta_c(x) \neq \eta_c(y)$ whenever $x \sim y$ in $\mathcal{G}_c$. For each $c \in \N$, we fix one such $\eta_c$.
		
		Given a sequence $s = (s_0, s_1, \ldots) \in \N^\N$, define a partial function $f_s \colon X \rightharpoonup \N$ as follows:
		\[
			f_s(x) \defeq \text{the smallest } c \in \N \text{ such that } \eta_c(x) = s_c, \text{ if such exists}.
		\]
		Note that for any $s \in \N^\N$, \eqref{eq:far} is satisfied with $f = f_s$. Indeed, if $x$, $y \in X$ are distinct and such that $f_s(x) = f_s(y) = c$, then, by definition, $\eta_c(x) = \eta_c(y) = s_c$, so $x \not \sim y$ in $\mathcal{G}_c$, i.e., $\dist(x, y) > d_c$. As the map~$f_s$ is Borel, it remains to prove that for some $s \in \N^\N$, the set
		\[
			\set{x \in X \,:\, f_s(x) \text{ is defined}}
		\]
		is comeager. Due to the Kuratowski--Ulam theorem~\cite[Theorem~8.41]{Kec95}, it suffices to show that for all $x \in X$, the set
		\[
			\set{s \in \N^\N \,:\, f_s(x) \text{ is defined}} \,=\, \set{s \in \N^\N \,:\, s_c = \eta_c(x) \text{ for some } c \in \N}
		\]
		is comeager in $\N^\N$, which is indeed the case as it is open and dense.
	\end{proof}
	
	Now we combine Lemmas~\ref{lemma:infty} and~\ref{lemma:MU} to prove the main technical result of this subsection:
	
	\begin{lemma}\label{lemma:cont}
		There exist a nonempty compact metrizable space $H$ with no isolated points, a dense countable subset $H_0 \subset H$, and a continuous map $H \to \S_0(\G, \N) \colon h \mapsto \subshift_h$ such that
		\begin{itemize}
			\item[--] for all $h \in H_0$, the subshift $\subshift_h$ is hard; and
			\item[--] for all $h \in H \setminus H_0$, the subshift $\subshift_h$ is easy.
		\end{itemize}
	\end{lemma}
	\begin{proof}
		Let $\N \cup \set{\infty}$ be the compactification of $\N$ obtained by adding the point $\infty$ so that the neighborhood filter of $\infty$ is generated by the sets
		$
		\set{n \in \N \cup \set{\infty} \,:\, n \geq m}
		$
		with $m$ ranging over $\N$. 
		Let~$H$ denote the set of all \emph{nondecreasing} sequences in $(\N \cup \set{\infty})^\N$, which we write as $h = (h_0, h_1, \ldots)$. Being a closed subset of a compact metrizable space, $H$ itself is compact and metrizable, and it is easy to see that~$H$ contains no isolated points. Let
		\[
			H_0 \defeq \set{h \in H \,:\, h_c = \infty \text{ for some } c \in \N}.
		\]
		Evidently, $H_0$ is a dense countable subset of $H$.
		
		Fix a sequence $(d_0, d_1, \ldots) \in \Rpos^\N$ such that a ball of radius $d_0$ in $\G$ contains at least~$2$ elements, and for each $c \in \N$, a ball of radius $d_{c+1}$ in $\G$ contains two disjoint balls of radius $d_c$ (such a sequence exists since $\G$ is infinite, while every ball of finite radius in $\G$ is finite). For $h \in H$, let $\subshift_h$ denote the set of all colorings $\omega \colon \G \to \N$ such that for all $c \in \N$, 
		\begin{equation*}
		\inf \set{\dist(\gamma, \delta) \,:\, \gamma, \delta \in \G, \, \gamma \neq \delta, \, \omega(\gamma) = \omega(\delta) = c} \, \geq \, \max\set{2d_c+1, h_c}.
		\end{equation*}
		The set $\subshift_h$ is a subshift; furthermore, the map $h \mapsto \subshift_h$ is continuous, since determining whether given $\phi \in \finfun{\G}{\N}$ belongs to $\FPC{\subshift_h}$ only involves checking bounds on $h_c$ for finitely many colors $c \in \N$.
		
		If $h \in H \setminus H_0$, then $\subshift_h$ is easy by Lemma~\ref{lemma:MU}. Now suppose $h \in H_0$ and consider any $\omega \in \subshift_h$. By Lemma~\ref{lemma:infty}, the set $\set{\omega(\gamma) \,:\, \gamma \in \G}$ is infinite. As $h \in H_0$, all but finitely many entries in $h$ are equal to $\infty$; therefore, there is some $c \in \N$ such that $h_c = \infty$ and $\omega(\gamma) = c$ for some $\gamma \in \G$. Since $\omega \in \subshift_h$, if $h_c = \infty$, then there is at most a single element $\gamma \in \G$ with $\omega(\gamma) = c$. Therefore, $\subshift_h$ is hard by Lemma~\ref{lemma:hard}.
	\end{proof}
	
	\subsection{The space of compact sets and a final reduction}
	
	The last step in our argument is inspired by the dichotomy theorem for co-analytic $\sigma$-ideals of compact sets due to Kechris, Louveau, and Woodin~\cite[Theorem~33.3]{Kec95}, which asserts that such a $\sigma$-ideal is either $G_\delta$, or else, complete co-analytic. The Kechris--Louveau--Woodin dichotomy theorem is proved using a result of Hurewicz (see Theorem~\ref{theo:Hurewicz} below), which we will utilize in much the same way in our proof of Theorem~\ref{theo:complete}.
	
	Before stating Hurewicz's theorem, we need to introduce some notation and terminology. Let~$X$ be a Polish space. We use $\K(X)$ to denote the set of all compact subsets of $X$. The set $\K(X)$ is equipped with the \emph{Vietoris topology}, which is generated by the open sets of the form
	\[
		\set{C \in \K(X) \,:\,C \cap U \neq \0} \;\;\;\text{and}\;\;\;\set{C \in \K(X) \,:\,C \subseteq U},
	\]
	where $U$ is ranging over the open subsets of $X$. The space $\K(X)$ is itself Polish~\cite[Theorem~4.25]{Kec95}. For more background on the Vietoris topology and related concepts, see \cite[Section~4.F]{Kec95} and \cite[Section~3.D]{Tse16}.
	
	\begin{theo}[{Hurewicz~\cite[Exercise~27.4(ii)]{Kec95}}]\label{theo:Hurewicz}
		Let $X$ be a Polish space and let $A \subseteq X$ be a subset which is $G_\delta$ but not $F_\sigma$. Then the set $\set{C \in \K(X) \,:\, C \cap A \neq \0}$	is complete analytic.
	\end{theo}
	
	\begin{lemma}\label{lemma:union}
		If $C \subseteq \S_0(\G, \N)$ is a compact set, then $\bigcup_{\subshift \in C} \subshift$ is a subshift. Furthermore, the map
		\[
			\K(\S_0(\G, \N)) \to \S_0(\G, \N) \colon C \mapsto \textstyle\bigcup\nolimits_{\subshift \in C} \subshift
		\]
		is continuous.
	\end{lemma}
	\begin{proof}
		Let $C \in \K(\S_0(\G, \N))$. The set $\bigcup_{\subshift \in C} \subshift$ is clearly shift-invariant. Consider any $\omega \in \overline{\bigcup_{\subshift \in C} \subshift}$ (the bar denotes topological closure). There exist a sequence of subshifts $\subshift_0$, $\subshift_1$, $\ldots \in C$ and a sequence of colorings $\omega_0 \in \subshift_0$, $\omega_1 \in \subshift_1$, \ldots{} such that $\lim_{i \to \infty} \omega_i = \omega$. Since $C$ is compact, we may pass to a subsequence and assume that the sequence $\subshift_0$, $\subshift_1$, \ldots{} converges to a limit $\subshift_\infty \in C$. Consider any $S \in \finset{\G}$ and let $\phi \defeq \rest{\omega}{S}$. As $\omega = \lim_{i \to \infty} \omega_i$, we have
		\[
			\phi = \rest{\omega_i}{S} \text{ for all sufficiently large } i \in \N.
		\]
		This implies
		\[
			\phi \in \FPC{\subshift_i} \text{ for all sufficiently large } i \in \N,
		\]
		and thus, $\phi \in \FPC{\subshift_\infty}$. Therefore, $\omega \in \subshift_\infty$, and hence,
		\[
			\textstyle\bigcup\nolimits_{\subshift \in C} \subshift \,=\, \overline{\textstyle\bigcup\nolimits_{\subshift \in C} \subshift},
		\]
		i.e., the set $\bigcup_{\subshift \in C} \subshift$ is closed, and hence, it is a subshift. The continuity of the map $C \mapsto \bigcup_{\subshift \in C} \subshift$ then follows immediately from the definitions of the topologies on $\S_0(\G, \N)$ and $\K(\S_0(\G, \N))$.
	\end{proof}
	
	Now we have all the necessary tools to finish the proof of Theorem~\ref{theo:complete}.
	
	\begin{proof}[Proof of Theorem~\ref{theo:complete}]
		Let $\alpha \colon \G \acts X$ be a free continuous action of $\G$ on a nonempty Polish space~$X$. As observed in Corollary~\ref{corl:analytic}, the set $\S_\BM(\alpha, \N)$ is analytic. The case of generically smooth $\alpha$ is handled in Lemma~\ref{lemma:smooth}, so it remains to show that if $\alpha$ is not generically smooth, then $\S_\BM(\alpha, \N)$ is complete.
		
		Let $H$ and $H_0$ be as in Lemma~\ref{lemma:cont} and let $H \to \S_0(\G, \N) \colon h \mapsto \subshift_h$ be a continuous function such that
		\begin{itemize}
			\item[--] for all $h \in H_0$, the subshift $\subshift_h$ is hard; and
			\item[--] for all $h \in H \setminus H_0$, the subshift $\subshift_h$ is easy.
		\end{itemize}
		Since continuous images of compact spaces are compact, for each $C \in \K(H)$, we have
		\[
			\set{\subshift_h \,:\, h \in C} \in \K(\S_0(\G, \N));
		\]
		moreover, the map
		\[
			\K(H) \to \K(\S_0(\G, \N)) \colon C \mapsto \set{\subshift_h \,:\, h \in C} 
		\]
		is continuous. Using Lemma~\ref{lemma:union}, we can then define a continuous function $\K(H) \to \S_0(\G, \N)$ by sending each $C \in \K(H)$ to the subshift $\subshift_C \defeq \bigcup_{h \in C} \subshift_h$. Notice that if $C \cap (H \setminus H_0) \neq \0$, then $\subshift_C \supseteq \subshift_h$ for some $h \in H\setminus H_0$, so $\subshift_C$ is easy and, in particular, $\subshift_C \in \S_\BM(\alpha, \N)$. On the other hand, if  $C \cap (H \setminus H_0) = \0$, i.e., if $C \subseteq H_0$, then $\subshift_C$ is a union of countably many hard subshifts, so, by Lemma~\ref{lemma:hard_union}, it is itself hard; since $\alpha$ is not generically smooth, this implies $\subshift_C \not \in \S_\BM(\alpha, \N)$. Therefore,
		\begin{equation}\label{eq:reduction}
			C \cap (H \setminus H_0) \neq \0 \, \Longleftrightarrow \, \subshift_C \in \S_\BM(\alpha, \N).
		\end{equation}
		Since $H \setminus H_0$ is the complement of a dense countable subset of a nonempty Polish space $H$ with no isolated points, it is $G_\delta$ but not $F_\sigma$; thus, by Theorem~\ref{theo:Hurewicz}, the set
		\begin{equation}\label{eq:Hur}
			\set{C \in \K(H) \,:\, C \cap (H \setminus H_0) \neq \0}
		\end{equation}
		is complete analytic. It remains to notice that, by~\eqref{eq:reduction}, the map $C \mapsto \subshift_C$ is a continuous reduction of the complete analytic set~\eqref{eq:Hur} to $\S_\BM(\alpha, \N)$.
	\end{proof}

	\section{Proof of Theorem~\ref{theo:combi}}\label{sec:combi}
	
	\noindent We break proving Theorem~\ref{theo:combi} up into three steps, each corresponting to one of the implications
	\[\text{\ref{item:combi:BM} $\Longrightarrow$ \ref{item:combi:amalgamation} $\Longrightarrow$ \ref{item:combi:local} $\Longrightarrow$ \ref{item:combi:BM}.}\]
	Using observations made in Section~\ref{sec:ideals}, we phrase and prove each implication in terms of $\G$\=/ideals rather than subshifts. Finally, we deduce Corollary~\ref{corl:not_universal} in~\secsign\ref{subsec:not_universal}.
	
	\subsection{Extendable $\G$-ideals with the join property from Baire measurable colorings}
	
	\begin{lemma}\label{lemma:col_to_join}
		Let $\P \subseteq \finfun{\G}{\N}$ be a $\G$-ideal. If $\sigma$ admits a Baire measurable $\P$-coloring, then there is an extendable $\G$-ideal $\P' \subseteq \P$ with the join property.
	\end{lemma}
	\begin{proof}
		Using Theorem~\ref{theo:gen_iso}, identify~$\sigma$ with the shift action $\sigma_\N \colon \G\acts \N^\G$. For $\phi \in \finfun{\G}{\N}$, let
		\[
		U_\phi \defeq \set{\omega \in \N^\G \,:\, \omega \supset \phi}.
		\]
		Note that $\set{U_\phi \,:\, \phi \in \finfun{\G}{\N}}$ is a basis for the topology on $\N^\G$ consisting of nonempty open (and closed) sets.
		
		Let $\P \subseteq \finfun{\G}{\N}$ be a $\G$-ideal and let $f \colon \N^\G \to \N$ be a Baire measurable $\P$-coloring of $\N^\G$. Set $\pi \defeq \pi_f$ and define
		\[
			\P' \defeq \set{\phi \in \finfun{\G}{\N} \,:\, \text{the set } \pi^{-1}(U_\phi) \text{ is nonmeager}}.
		\]
		It is clear that $\P'$ is a $\G$-ideal and, by the choice of $f$, we have $\P' \subseteq \P$.
		
		We claim that $\P'$ is extendable. Indeed, let $\phi \in \P'$ and $\gamma \in \G \setminus \dom(\phi)$. For each $c \in \N$, set $\phi_c \defeq \phi \cup \set{(\gamma, c)}$. We need to show that $\phi_c \in \P'$ for some $c \in \N$. To that end, notice that
			\begin{equation}\label{eq:split}
				\pi^{-1}(U_\phi) \,=\, \bigcup_{c \in \N} \pi^{-1}(U_{\phi_c}).
			\end{equation}
		Since $\phi \in \P'$, the set on the left-hand side of~\eqref{eq:split} is nonmeager; thus, at least one of the sets whose union is taken on the right-hand side of~\eqref{eq:split} must also be nonmeager, as desired.
		
		To finish the proof of the lemma, it remains to show that $\P'$ has the join property. Define an invariant map $R \colon \P' \to \Rpos$ as follows: For each $\phi \in \P'$, set $R(\phi)$ to be the smallest $R \in \N$ such that there is a map $\psi \colon \mathbf{Ball}(\dom(\phi), R) \to \N$ with $U_\psi \Vdash \pi^{-1}(U_\phi)$.
		(Such $R$ exists since the set $\pi^{-1}(U_\phi)$ is nonmeager.) Suppose that $\phi_1$, \ldots, $\phi_k \in \P'$ are pairwise $R$\=/separated and let $\phi \defeq \phi_1 \cup \ldots \cup \phi_k$. For each $1 \leq i \leq k$, choose $\psi_i \colon \mathbf{Ball}(\dom(\phi_i), R(\phi_i)) \to \N$ so that $U_{\psi_i} \Vdash \pi^{-1}(U_{\phi_i})$. Since $\phi_1$, \ldots, $\phi_k$ are pairwise $R$-separated, for all $i \neq j$, we have
		\[
			\dom(\psi_i) \cap \dom(\psi_j) \,=\, \mathbf{Ball}(\dom(\phi_i), R(\phi_i)) \cap \mathbf{Ball}(\dom(\phi_j), R(\phi_j)) \,=\, \0,
		\] 
		so $\psi \defeq \psi_1 \cup \ldots \cup \psi_k$ is a function in $\finfun{\G}{\N}$. Then
		\[
			U_\psi \,=\, U_{\psi_1} \cap \ldots \cap U_{\psi_k} \,\Vdash\, \pi^{-1}(U_{\phi_1}) \cap \ldots \cap \pi^{-1}(U_{\phi_k}) \,=\, \pi^{-1}(U_\phi).
		\]
		Therefore, the set $U_\phi$ is nonmeager, i.e., $\phi \in \P'$, as desired.
	\end{proof}
	
	\subsection{Reducing extendable $\G$-ideals with the join property to local ones}
	
	\begin{lemma}\label{lemma:red_to_loc}
		Every extendable $\G$-ideal with the join property is reducible to a local extendable $\G$-ideal.
	\end{lemma}
	\begin{proof}
		Let $\P \subseteq \finfun{\G}{\N}$ be an extendable $\G$-ideal with the join property and let $R \colon \P \to \Rpos$ be an invariant function such that whenever $k \in \N$ and $\phi_1$, \ldots, $\phi_k \in \P$ are pairwise $R$\=/separated, we have
		$\phi_1 \cup \ldots \cup \phi_k \in \P$.
		We may assume that $R$ is monotone increasing, i.e., for all $\phi$, $\phi' \in \P$,
		\[
			\phi \subseteq \phi' \,\Longrightarrow\, R(\phi) \leq R(\phi').
		\]
		Otherwise we can replace $R$ with the map $\tilde{R} \colon \P \to \Rpos$ defined by
		\[
			\tilde{R}(\phi) \defeq \sup \set{R(\phi') \,:\, \phi' \subseteq \phi}.
		\]
		
		We will explicitly construct a local extendable $\G$-ideal $\P'$ such that $\P \succeq \P'$. It will be more convenient to view $\P'$ as a subset of $\finfun{\G}{(\N \times \N)}$ rather than $\finfun{\G}{\N}$ (of course, we can turn it into a subset of $\finfun{\G}{\N}$ using an arbitrary bijection between $\N\times \N$ and $\N$). Let $\pi_1$, $\pi_2 \colon \N \times\N \to \N$ denote the projection maps:
		$$
			\pi_1(h, c) \defeq h \;\;\;\text{and}\;\;\; \pi_2(h, c) \defeq c\;\;\; \text{for all}\;\;\; (h, c) \in \N \times \N.
		$$
		Given $\phi \in \finfun{\G}{(\N \times \N)}$, an element $\gamma \in \G$, a radius $r \in \Rpos$, and a threshold $h \in \N$, define
		\[
			\phi[\gamma, r; h] \defeq \rest{\phi}{\set{\delta \in \dom(\phi) \cap \Ball{\gamma}{r} \,:\, (\pi_1 \circ \phi)(\delta) \leq h}}.
		\] 
		By definition, $\phi[\gamma, r; h] \subseteq \phi[\gamma, r]$. Let $\P'$ denote the set of all partial maps $\phi \in \finfun{\G}{(\N \times \N)}$ such that the following holds for all $\gamma \in \dom(\phi)$: If we let $h \defeq (\pi_1 \circ \phi)(\gamma)$ and $\psi \defeq \pi_2 \circ (\phi[\gamma, 3h; h])$, then
		\[
			\dom(\psi) \subseteq \Ball{\gamma}{h}; \;\;\;\;\;\;\; \psi \in \P; \;\;\;\;\;\;\; \text{and} \;\;\;\;\;\;\; R(\psi) \leq h.
		\]
		\noindent Evidently, $\P'$ is invariant under the action $\G \acts \finfun{\G}{(\N \times \N)}$. Moreover, since the map $R$ is monotone increasing, $\P'$ is closed under restrictions; in other words, $\P'$ is a $\G$-ideal. By definition,
		\[
			\P' = \loc{\P'}{r} \;\;\;\text{for}\;\;\; r \colon \N \times \N \to \Rpos \colon (h, c) \mapsto 3h.
		\]
		It remains to verify that $\P'$ is extendable and $\P\succeq\P'$.
		
		\begin{claim}\label{claim:partition}
			Let $\phi \in \P'$ and let
			\[
				h \defeq \sup \set{(\pi_1 \circ \phi)(\gamma) \,:\, \gamma \in \dom(\phi)}.
			\]
			Then $\phi$ can be written as a union 
			$\phi = \phi_1 \cup \ldots \cup \phi_k$ for some $k \in \N$ and $\phi_1$,~\ldots, $\phi_k \in \P'$ with the following properties:
			\begin{enumerate}[label=\normalfont{--}]
				\item for each $1 \leq i \leq k$, the map $\psi_i \defeq \pi_2 \circ \phi_i$ belongs to $\P$;
				\item for each $1 \leq i \leq k$, we have $R(\psi_i) \leq h$;
				\item the maps $\psi_1$, \ldots, $\psi_k$ are pairwise $R$-separated.
			\end{enumerate}
		\end{claim}
		\begin{claimproof}
			The proof is by induction on $|\dom(\phi)|$. If $\phi = \0$, then the claim holds vacuously with $k = 0$. Now suppose that $\phi \neq \0$. Then there is $\gamma_0 \in \dom(\phi)$ such that $(\pi_1 \circ \phi)(\gamma_0) = h$. Set
			\[
				\phi_0 \defeq \phi[\gamma_0, 3h]; \;\;\;\;\;\;\; \psi_0 \defeq \pi_2 \circ \phi_0; \;\;\;\;\;\;\;\text{and}\;\;\;\;\;\;\; \phi' \defeq \phi \setminus \phi_0.
			\]
			By the choice of $h$, we have $\phi[\gamma_0, 3h; h] = \phi_0$. Thus, by the definition of $\P'$,
			\[
				\dom(\psi_0) \subseteq \Ball{\gamma_0}{h}; \;\;\;\;\;\;\; \psi_0 \in \P; \;\;\;\;\;\;\;\text{and}\;\;\;\;\;\;\; R(\psi_0) \leq h.
			\]
			Applying the inductive hypothesis to $\phi'$, we can write
			$
				\phi' = \phi_1 \cup \ldots \cup \phi_k
			$
			for some $\phi_1$, \ldots, $\phi_k \in \P'$ with the following properties:
			\begin{itemize}
				\item[--] for each $1 \leq i \leq k$, the map $\psi_i \defeq \pi_2 \circ \phi_i$ belongs to $\P$;
				\item[--] for each $1 \leq i \leq k$, we have $R(\psi_i) \leq h$;
				\item[--] the maps $\psi_1$, \ldots, $\psi_k$ are pairwise $R$-separated.
			\end{itemize}
			It remains to show that $\psi_0$ is $R$-separated from each $\psi_i$ with $1 \leq i \leq k$. Suppose, towards a contradiction, that for some $1 \leq i \leq k$,
			\[
				\dist(\dom(\psi_0), \dom(\psi_i)) \,\leq\, R(\psi_0) + R(\psi_i) \,\leq\, 2h.
			\]
			Let $\gamma \in \dom(\psi_0)$ be such that $\dist(\gamma, \dom(\psi_i)) \leq 2h$. Since $\dom(\psi_0) \subseteq \Ball{\gamma_0}{h}$, we obtain
			\[
				\dist(\gamma_0, \dom(\psi_i)) \,\leq\, \dist(\gamma_0, \gamma) + \dist(\gamma, \dom(\psi_i)) \,\leq\, h + 2h \,=\, 3h.
			\]
			On the other hand, by construction, $\dom(\psi_i) \cap \Ball{\gamma_0}{3h} = \0$. This contradiction completes the proof.
		\end{claimproof}
		
		Consider any $\phi \in \P'$ and let $\phi = \phi_1 \cup \ldots \cup \phi_k$ be a decomposition of $\phi$ given by Claim~\ref{claim:partition}. For each $1 \leq i \leq k$, let $\psi_i \defeq \pi_2 \circ \phi_i$. Then every $\psi_i$ belongs to $\P$ and $\psi_1$, \ldots, $\psi_k$ are pairwise $R$-separated. By the choice of $R$, this yields
		\[
			\pi_2 \circ \phi \,=\, \psi_1 \cup \ldots \cup \psi_k \,\in\, \P.
		\]
		Therefore, $\pi_2$ is a reduction of $\P$ to $\P'$.
		
		Finally, to see that $\P'$ is extendable, let $\phi \in \P'$ and let $\gamma \in \G \setminus \dom(\phi)$. Set $\psi \defeq \pi_2 \circ \phi$. We already know that $\psi \in \P$. Since $\P$ is extendable, there is $c \in \N$ such that $\psi' \defeq \psi \cup \set{(\gamma, c)} \in \P$. Choose $h \in \N$ so large that the following statements are true:
		\[
			h \geq R(\psi'); \;\;\;\;\;\;\; h > (\pi_1 \circ \phi)(\delta) \text{ for all } \delta \in \dom(\phi);\;\;\;\;\;\;\; \text{and}\;\;\;\;\;\;\; \Ball{\gamma}{h} \supseteq \dom(\psi').
		\]
		Then $\phi \cup \set{(\gamma, (h, c))} \in \P'$, as desired.
	\end{proof}
	
	\subsection{Baire measurable colorings from locality and extendability}
	
	Before proceeding with the last part of the proof of Theorem~\ref{theo:combi}, we introduce some terminology and notation related to partial (but not necessarily finite) maps $\phi \colon \G \rightharpoonup \N$. The set of all such maps is denoted by $\pfun{\G}{\N}$. A partial map $\phi \colon \G \rightharpoonup \N$ can be viewed as a (total) function $\phi \colon \G \to \N \cup \set{\text{undefined}}$, where ``$\text{undefined}$'' is a special symbol distinct from all the elements of $\N$. In that way,
	\[
		\pfun{\G}{\N} \;\;\;\text{is the same as}\;\;\;	(\N \cup \set{\text{undefined}})^\G,
	\]
	and the latter set carries the product topology (the topology on $\N \cup \set{\text{undefined}}$ is discrete) and is equipped with the shift action of $\G$. Note that $\finfun{\G}{\N}$ is a countable dense subset of $\pfun{\G}{\N}$ and $\N^\G$ is a closed subset of $\pfun{\G}{\N}$.
	
	Similarly to the notation we use for finite partial functions, given $\phi \in \pfun{\G}{\N}$, an element $\gamma \in \G$, and a radius $r \in \Rpos$, let
	\[
		\phi[\gamma, r] \defeq \rest{\phi}{(\dom(\phi) \cap \Ball{\gamma}{r})}.
	\]
	By definition, $\phi[\gamma, r] \in \finfun{\G}{\N}$.
	
	Let $\P \subseteq \finfun{\G}{\N}$ be a $\G$-ideal. A \emph{partial $\P$-coloring} is a map $\phi \in \pfun{\G}{\N}$ such that
	\[
		\rest{\phi}{S} \in \P\;\;\; \text{for all} \;\;\;S \in \finset{\dom(\phi)}.
	\]
	The set of all partial $\P$-colorings is denoted by $\pcol{\G}{\N}{\P}$. Note that
	\[
		\pcol{\G}{\N}{\P} \, \cap\, \finfun{\G}{\N} \,=\, \P \;\;\;\;\;\text{and}\;\;\;\;\; \pcol{\G}{\N}{\P} \, \cap\, \N^\G \,=\, \Sol(\P).
	\]
	If $\P$ is local and $r \colon \N \to \Rpos$ is a function such that $\P = \loc{\P}{r}$, then for all $\phi \in \pfun{\G}{\N}$,
	\begin{equation}\label{eq:using_local}
	\phi \in \pcol{\G}{\N}{\P} \;\;\;\Longleftrightarrow\;\;\; \phi[\gamma, r(\phi(\gamma))] \in \P \text{ for all } \gamma \in \dom(\phi).
	\end{equation}
	
	\begin{lemma}\label{lemma:loc_to_col}
		If $\P$ is a local extendable $\G$-ideal, then $\sigma$ admits a Baire measurable $\P$\=/coloring.
	\end{lemma}
	\begin{proof}
		Let $\P \subseteq \finfun{\G}{\N}$ be a local extendable $\G$-ideal and let $r \colon \N \to \Rpos$ be a function such that $\P = \loc{\P}{r}$.
		
		Using Theorem~\ref{theo:gen_iso}, identify $\sigma$ with the shift action $\sigma_{2^\N} \colon \G \acts (2^\N)^\G$ and then replace it by the product action $(\sigma_2)^\N \colon \G \acts (2^\G)^\N$ (the spaces $(2^\N)^\G$ and $(2^\G)^\N$ are equi\-va\-riantly homeomorphic).
		We will find an equivariant Borel map $\pi \colon (2^\G)^\N \to \pcol{\G}{\N}{\P}$ such that the set $\pi^{-1}(\N^\G)$ is comeager. This will yield the desired result since given such $\pi$, any function $f \colon (2^\G)^\N \to \N$ with $f(x) = \pi(x)(\mathbf{1})$ for all $x \in \pi^{-1}(\N^\G)$ is a Baire measurable $\P$-coloring of $(2^\G)^\N$.

		For $x \in 2^\G$, the \emph{support} of $x$ is the set
		\[
			\supp(x) \defeq \set{\gamma \in \G \,:\, x(\gamma) = 1}.
		\]
		Set $X \defeq (2^\G)^\N$. We write the elements of $X$ as sequences of the form $x = (x_0, x_1, \ldots)$.
		
		Fix a sequence $(c_0, c_1, \ldots) \in \N^\N$ in which every $c \in \N$ appears infinitely many times and set
		\[
			R_i \defeq \sup\set{r(c_0), \ldots, r(c_{i-1})}.
		\]
		For each $x \in X$, define a sequence of partial maps $\pi_i(x) \in \pfun{\G}{\N}$ inductively as follows:
		
		\begin{leftbar}
			\noindent {\sc Step $0$:} Set $\pi_0(x) \defeq \0$.
			
			\noindent {\sc Step $i+1$:} Let $S_i(x)$ denote the set of all $\gamma \in \G$ such that
			\begin{itemize}
				\item[--] $\pi_i(x)(\gamma)$ is not defined;
				\item[--] $\Ball{\gamma}{2R_i} \cap \supp(x_i) = \set{\gamma}$; and
				\item[--] $\pi_i(x)[\gamma, 2R_i]  \cup  \set{(\gamma, c_i)} \in\P$.
			\end{itemize}
			For all $\gamma \in \G$, set
			\[
			\pi_{i+1}(x)(\gamma) \defeq \begin{cases}
			\pi_i(x)(\gamma) &\text{if } \pi_i(x)(\gamma) \text{ is defined};\\
			c_i &\text{if }\gamma \in S_i(x);\\
			\text{undefined} &\text{otherwise}.
			\end{cases}
			\]
		\end{leftbar}

		\noindent By construction, for all $x \in X$, we have
		\[
		\0 = \pi_0(x) \subseteq \pi_1(x) \subseteq \ldots,
		\]
		so we can define $\pi_\infty(x) \in \pfun{\G}{\N}$ via
		\[
		\pi_\infty(x) \defeq \bigcup_{i=0}^\infty \pi_i(x).
		\]
		It is clear that the maps $\pi_i \colon X \to \pfun{\G}{\N}$ are equivariant. Notice that they are also continuous. Indeed, the value $\pi_i(x)(\gamma)$---including whether or not it is defined---is determined by the restrictions of the first $i$ functions $x_0$, \ldots, $x_{i-1}$ to the finite set $\Ball{\gamma}{2R_0 + \cdots + 2R_{i-1}}$. Being a pointwise limit of equivariant continuous functions, the map $\pi_\infty \colon X \to \pfun{\G}{\N}$ is equivariant and Borel.
		
		\begin{claim}\label{claim:proper}
			For all $x \in X$, we have $\pi_\infty(x) \in \pcol{\G}{\N}{\P}$.
		\end{claim}
		\begin{claimproof}
			Let $x \in X$. Since $\pi_\infty(x)$ is the union of the increasing sequence $\pi_0(x) \subseteq \pi_1(x) \subseteq \ldots$, it is sufficient (and necessary) to establish that $\pi_i(x) \in \pcol{\G}{\N}{\P}$ for all $i \in \N$. We proceed by induction on~$i$. The base case is trivial since $\pi_0(x) = \0 \in \P$ by definition (recall that $\P$ is local, hence nonempty). Now suppose that $\pi_i(x) \in \pcol{\G}{\N}{\P}$ and consider the partial map $\pi_{i+1}(x)$. By~\eqref{eq:using_local}, it is enough to show that for all $\gamma \in \dom(\pi_{i+1}(x))$,
			\[
				\pi_{i+1}(x)[\gamma, r(\pi_{i+1}(x)(\gamma))] \, \in \, \P.
			\]
			By construction, $\pi_{i+1}(x)$ takes values in the set $\set{c_0, \ldots, c_i}$. Therefore,
			\[
				r(\pi_{i+1}(x)(\gamma)) \leq R_i \;\;\;\text{for all}\;\;\; \gamma \in \dom(\pi_{i+1}(x)).
			\]
			Thus, it suffices to prove that for all $\gamma \in \dom(\pi_{i+1}(x))$,
			\[
				\pi_{i+1}(x)[\gamma, R_i] \, \in \, \P.
			\]
			Consider any $\gamma \in \dom(\pi_{i+1}(x))$. If $\Ball{\gamma}{R_i} \cap S_i(x) = \0$, then
			\[
				\pi_{i+1}(x)[\gamma, R_i] \, = \, \pi_{i}(x)[\gamma, R_i] \,\in\, \P
			\]
			by the inductive hypothesis. Now assume that $\delta \in \Ball{\gamma}{R_i} \cap S_i(x)$. Then $\Ball{\gamma}{R_i} \subseteq \Ball{\delta}{2R_i}$, so it is enough to show
			\[
				\pi_{i+1}(x)[\delta, 2R_i] \, \in \, \P.
			\]
			As $\delta \in S_i(x)$, we have $\pi_{i+1}(x)(\delta) = c_i$ and
			\[
				\delta \,\in\, \Ball{\delta}{2R_i} \,\cap\, S_i(x) \, \subseteq \,\Ball{\delta}{2R_i} \cap \supp(x_i) = \set{\delta},
			\]
			which implies
			\[
				\pi_{i+1}(x)[\delta, 2R_i] \,=\, \pi_{i}(x)[\delta, 2R_i] \,\cup\, \set{(\delta, c_i)} \,\in\, \P.\qedhere
			\]
		\end{claimproof}
		
		Thus, the above construction produces an equivariant Borel map $\pi_\infty \colon X \to \pcol{\G}{\N}{\P}$. To finish the argument, it remains to show that the set $(\pi_\infty)^{-1}(\N^\G)$ is comeager. We have
		\begin{align*}
			(\pi_\infty)^{-1}(\N^\G) \,&=\, \set{x \in X \,:\, \pi_\infty(x)(\gamma) \text{ is defined for all } \gamma \in \G} \,\\
			&=\, \bigcap_{\gamma \in \G} \set{x \in X \,:\, \pi_\infty(x)(\gamma) \text{ is defined}},
		\end{align*}
		so we only need to verify that for each $\gamma \in \G$, the set $\set{x \in X \,:\, \pi_\infty(x)(\gamma) \text{ is defined}}$ is comeager. To that end, consider any $\gamma \in \G$ and write
		\begin{equation*}\label{eq:steps}
			\set{x \in X \,:\, \pi_\infty(x)(\gamma) \text{ is defined}} \,=\, \bigcup_{i=0}^\infty \set{x \in X \,:\, \pi_i(x)(\gamma) \text{ is defined}}.
		\end{equation*}
		By the continuity of $\pi_i$ for all $i \in \N$, the sets $\set{x \in X \,:\, \pi_i(x)(\gamma) \text{ is defined}}$ are open. Therefore, their union is open as well; it remains to show that it is dense. Let $U \subseteq X$ be a nonempty open subset. We need to find an element $x \in U$ such that $\pi_\infty(x)(\gamma)$ is defined. By passing to a smaller open subset if necessary, we may assume that $U$ is of the form
		\[
			U \,=\, U_0 \times \cdots \times U_{i-1} \times 2^\G \times 2^\G \cdots,
		\]
		for some nonempty open subsets $U_0$, \ldots, $U_{i-1} \subseteq 2^\G$. Notice that the set of all functions $\G \to 2$ with finite support is dense in $2^\G$; therefore, for each $0 \leq k < i$, we can choose $y_k \in U_k$ so that $\supp(y_k)$ is finite. Let
		\[
			A \defeq \set{x \in X \,:\, x_k = y_k \text{ for all } 0 \leq k < i}.
		\]
		By the choice of $y_0$, \ldots, $y_{i-1}$, we have $\0 \neq A \subseteq U$. Since for all $x \in X$, the value $\pi_i(x)$ is determined by the first $i$ functions $x_0$, \ldots, $x_{i-1}$, we can define $\phi \in \pcol{\G}{\N}{\P}$ by
		\[
			\phi \defeq \pi_i(x) \;\;\; \text{for some (hence all)}\;\;\; x \in A.
		\]
		If $\gamma \in \dom(\phi)$, then $\pi_\infty(x)(\gamma)$ is defined for all $x \in A$ and we are done, so assume that $\gamma \not \in \dom(\phi)$. Since
		\[
			\dom(\phi) \,\subseteq\, \supp(x_0) \cup \ldots \cup \supp(x_{i-1}),
		\]
		the domain of $\phi$ is finite, i.e., $\phi \in \P$. The $\G$-ideal $\P$ is extendable, so there is $c \in \N$ such that
		\[
			\psi \,\defeq\, \phi \cup \set{(\gamma, c)} \,\in\, \P.
		\] By the choice of the sequence $(c_0, c_1, \ldots)$, there is some index $j \geq i$ such that $c_j = c$. For all $i \leq k < j$, set $y_k \colon \G \to 2$ to be the constant $0$ function, and set $y_j \colon \G \to 2$ to be the characteristic function of the one-element set $\set{\gamma}$.
		Let
		\[
			B \defeq \set{x \in X \,:\, x_k = y_k \text{ for all } 0 \leq k \leq j}.
		\]
		Then $\0 \neq B \subseteq A$, and for all $x \in B$, we have $\pi_{j+1}(x) = \psi$, in particular, $\pi_\infty(x)(\gamma)$ is defined.
	\end{proof}
	
	\subsection{Proof of Corollary~\ref{corl:not_universal}}\label{subsec:not_universal}
	
	A continuous action $\alpha$ of $\G$ on a compact space is \emph{minimal} if every $\alpha$-orbit is dense. It follows from a result of Gao, Jackson, and Seward~\cite[Theorem~1.4.1]{GJS16} that $\G$ admits a free minimal action on a nonempty compact metrizable space.
	
	Recall the following notation, introduced in~\secsign\ref{subsec:com_lem}: For a free action $\alpha \colon \G \acts X$ and $x$, $y \in X$, write
	\[
	\dist(x, y) \defeq \begin{cases}
	\dist(\mathbf{1}, \gamma) &\text{if } \gamma \in \G \text{ is such that } \gamma \cdot x = y;\\
	\infty &\text{if } x \text{ and } y \text{ are in different $\alpha$-orbits}.
	\end{cases}
	\]
	Given $x \in X$ and $A \subseteq X$, let
	\[
		\dist(x, A) \defeq \inf\set{\dist(x, y) \,:\, y \in A}.
	\]
	
	\begin{lemma}\label{lemma:req}
		Let $\alpha \colon \G \acts X$ be a free minimal action of $\G$ on a nonempty compact metrizable space $X$. Suppose that $A \subseteq X$ is a nonmeager Baire measurable set. Then there exists a radius $R \in \Rpos$ such that the set $\set{x \in X \,:\, \dist(x, A) \leq R}$ is comeager.
	\end{lemma}
	\begin{proof}
		By the Baire alternative, there is nonempty open $U \subseteq X$ such that $U \Vdash A$. Since $\alpha$ is minimal, we have $X = \bigcup_{\gamma \in \G} (\gamma \cdot U)$. As $X$ is compact, there is a radius $R \in \Rpos$ such that
		\[
			X = \textstyle\bigcup\nolimits_{\gamma \in \Ball{\mathbf{1}}{R}} (\gamma \cdot U).
		\]
		Therefore, the set
		\[
			X' \defeq \textstyle\bigcup\nolimits_{\gamma \in \Ball{\mathbf{1}}{R}} (\gamma \cdot A)
		\]
		is comeager. It remains to notice that $\dist(x, A) \leq R$ for all $x \in X'$.
	\end{proof}
	
	\begin{proof}[Proof of Corollary~\ref{corl:not_universal}]
		Let $\alpha \colon \G \acts X$ be any free minimal action of $\G$ on a nonempty compact metrizable space $X$. We will explicitly construct a $\G$-ideal $\P \subseteq \finfun{\G}{\N}$ such that $\sigma$ admits a Baire measurable $\P$-coloring, while $\alpha$ does not.
		
		Fix a sequence $(d_0, d_1, \ldots) \in \Rpos^\N$ such that a ball of radius $d_0$ in $\G$ contains at least~$2$ elements, and for each $c \in \N$, a ball of radius $d_{c+1}$ in $\G$ contains two disjoint balls of radius $d_c$ (such a sequence exists since $\G$ is infinite, while every ball of finite radius in $\G$ is finite). For each $c \in \N$, choose $D_c \in \Rpos$ so that the set $\set{\gamma \in \G \,:\, 2d_c < \dist(\mathbf{1}, \gamma) \leq D_c}$ contains a ball of radius~$d_c$.
		
		Let $\P$ denote the set of all partial maps $\phi \in \finfun{\G}{\N}$ such that the following holds for all $\gamma \in \dom(\phi)$: If we let $c \defeq \phi(\gamma)$, then for all $\delta \in \dom(\phi)$,
		\begin{enumerate}[label=\normalfont{(\arabic*)}]
			\item\label{item:corl_far} if $\dist(\gamma, \delta) \leq 2d_c$, then $\phi(\delta) \neq c$;
			\item if $2d_c < \dist(\gamma, \delta) \leq D_c$, then $\phi(\delta) > c$.
		\end{enumerate}
		\noindent Clearly, $\P$ is a $\G$-ideal. By definition, we have
		\[
			\P = \loc{\P}{r} \;\;\;\text{for}\;\;\; r \colon \N \to \Rpos \colon c \mapsto D_c,
		\]
		so $\P$ is local. Consider any $\phi \in \P$ and $\gamma \in \G \setminus \dom(\phi)$. Choose $c \in \N$ so large that the following statements are true:
		\[
			c > \phi(\delta) \text{ for all } \delta \in \dom(\phi) \;\;\;\;\;\;\;\text{and}\;\;\;\;\;\;\; \Ball{\gamma}{2d_c} \supseteq \dom(\phi).
		\]
		Then $\phi \cup \set{(\gamma, c)} \in \P$. This shows that $\P$ is extendable. Using Theorem~\ref{theo:combi}, we then conclude that $\sigma$ admits a Baire measurable $\P$-coloring.
		
		Now suppose, towards a contradiction, that $f \colon X \to \N$ is a Baire measurable $\P$-coloring of $\alpha$. Let $c_0 \in \N$ be any color such that the set $A \defeq f^{-1}(c_0)$ is nonmeager. By Lemma~\ref{lemma:req}, there is a radius $R \in \Rpos$ such that the set $\set{x \in X \,:\, \dist(x, A) \leq R}$ is comeager. Since the set $(\pi_f)^{-1}(\Sol(\P))$ is also comeager, we can choose $x \in X$ so that
		\[
			\pi_f(x) \in \Sol(\P) \;\;\;\;\;\;\;\text{and}\;\;\;\;\;\;\; \dist(\gamma \cdot x, A) \leq R \text{ for all } \gamma \in \G.
		\]
		Let $\omega \defeq \pi_f(x)$. Since $\omega \in \Sol(\P)$, Lemma~\ref{lemma:infty} implies that the set $\set{\omega(\gamma)\,:\, \gamma \in \G}$ is infinite; in particular, it contains an element $c$ such that $c \geq c_0$ and $d_c \geq R$. Take any $\gamma \in \G$ with $\omega(\gamma) = c$. By the choice of~$D_c$, there is some $\delta \in \G$ satisfying
		\[
			\Ball{\delta}{R} \,\subseteq\, \Ball{\delta}{d_c} \,\subseteq\, \set{\epsilon \in \G \,:\, 2d_c < \dist(\gamma, \epsilon) \leq D_c}.
		\]
		Since $\omega \in \Sol(\P)$, we have $\omega(\epsilon) > c$ for all $\epsilon \in \Ball{\delta}{R}$; in particular, there is no $\epsilon \in \Ball{\delta}{R}$ with $\omega(\epsilon) = c_0$. But then $\dist(\delta \cdot x, A) > R$; a contradiction.
	\end{proof}
	
	\medskip
	
	{
	\subsection*{Acknowledgements} I am grateful to Anush Tserunyan for insightful conversations and encouragement, and to the anonymous referee for helpful suggestions.
	}
	
	\printbibliography
	
\end{document}